\documentclass[a4paper,11pt]{article}

\usepackage{amscd}
\usepackage{amsfonts}
\usepackage{amsmath}
\usepackage{amssymb}
\usepackage{amsthm}
\usepackage{ascmac}
\usepackage{mathrsfs}
\usepackage{txfonts}

\allowdisplaybreaks
\makeindex

\newcommand{\N}{\mathbb{N}}

\newcommand{\R}{\mathbb{R}}
\newcommand{\Z}{\mathbb{Z}}

\newcommand{\A}{\mathscr{A}}
\newcommand{\B}{\mathscr{B}}

\newcommand{\M}{\mathscr{M}}

\newcommand{\Opn}{\mathscr{O}}

\newcommand{\Sym}{\mathscr{S}}
\newcommand{\T}{\mathscr{T}}
\newcommand{\U}{\mathscr{U}}
\newcommand{\V}{\mathscr{V}}
\newcommand{\W}{\mathscr{W}}
\newcommand{\X}{\mathscr{X}}

\theoremstyle{plain}
\newtheorem{thm}{Theorem}[section]
\newtheorem{lmm}[thm]{Lemma}
\newtheorem{prp}[thm]{Proposition}
\newtheorem{crl}[thm]{Corollary}
\theoremstyle{definition}
\newtheorem{dfn}[thm]{Definition}

\newtheorem{qst}[thm]{Question}

\def\ens#1{\{ #1 \}}
\def\Ens#1{\left\{ #1 \right\}}
\def\set#1#2{\{ #1 \mid #2 \}}
\def\Set#1#2{\left\{ #1 \ \middle|\ #2 \right\}}
\def\t#1{\text{\rm #1}}
\def\m#1{\text{#1}}
\def\v#1{| #1 |}
\def\vv#1{\left| #1 \right|}
\def\n#1{\| #1 \|}
\def\nn#1{\left\| #1 \right\|}

\title{
On Tate's Acyclicity and Uniformity of\\
Berkovich Spectra and Adic Spectra
}
\author{Tomoki Mihara}
\date{}

\begin{document}

\maketitle
\begin{abstract}
We construct a non-sheafy uniform Banach algebra such that a rational localisation of the Berkovich spectrum does not preserve the uniformity. We also construct uniform affinoid rings in the sense of Roland Huber such that rational localisations of the adic spectra do not preserve the uniformity. One of them is an example of a non-sheafy uniform affinoid ring. We introduce the notion of local uniformity instead, and prove that the local uniformity implies the sheaf condition.
\end{abstract}

\tableofcontents

\section{Introduction}
\label{Introsuction}

Throughout this paper, let $k$ denote a complete valuation field with non-trivial valuation of height $1$. An affinoid ring $(\A^{\triangleright},\A^{+})$ is said to be {\it sheafy} if the structure presheaf on the topology of $\t{Spa}(\A^{\triangleright},\A^{+})$ is a sheaf. We always assume that $\A^{\triangleright}$ is a complete Tate ring, and mainly consider the case where $\A^{\triangleright}$ is a topological $k$-algebra. Similarly to an affinoid ring, a Banach $k$-algebra $(\A,\n{\cdot}_{\A})$ is said to be {\it sheafy} if the Banach algebra of sections on a rational domain of the Berkovich spectrum is independent of the choice of its presentation, and if the structure presheaf on the Grothendieck topology generated by rational domains is a sheaf. Since the universality of an affinoid domain does not necessarily hold, the Banach $k$-algebra of sections associated to a rational domain of a Berkovich spectrum deeply depends on the choice of its presentation in general. This is one of the essential differences between a Berkovich spectrum and an adic spectrum.

\vspace{0.2in}
A Banach $k$-algebra $(\A,\n{\cdot})$ (resp.\ An affinoid ring $(\A^{\triangleright},\A^+)$ over $k$) is said to be {\it uniform} if the equality $\n{f^2} = \n{f}^2$ holds for any $f \in \A$ (resp.\ if $(\A^{\triangleright})^{\circ}$ is bounded). We construct an example of a uniform Banach algebra such that a rational localisation of the Berkovich spectrum does not preserve the uniformity. It is an example of a non-sheafy uniform Banach algebra. We also construct examples of uniform affinoid rings such that rational localisations of the adic spectra do not preserve the uniformity. One of them is an example of a non-sheafy uniform affinoid ring. We explain briefly the motivation of constructing such examples.

\vspace{0.2in}
First of all, we are interested in whether there is a good criterion of the sheaf condition. We recall several known facts for it.

\begin{thm}[\cite{Tat} Theorem 8.2]
\label{Tate}
Every affinoid $K$-algebra is sheafy.
\end{thm}

\begin{thm}[\cite{Ber} Proposition 2.2.5]
\label{Berkovich}
Every $K$-affinoid algebra is sheafy.
\end{thm}

\begin{thm}[\cite{Hub2} Theorem 2.2]
\label{Huber}
Every strongly Noetherian Tate ring is sheafy.
\end{thm}

\begin{thm}[\cite{Sch} Theorem 6.3 (iii)]
\label{Scholze}
Every perfectoid affinoid $K$-algebra is sheafy.
\end{thm}

Theorem \ref{Tate} and Theorem \ref{Berkovich} are results for Berkovich spectra, and Theorem \ref{Huber} and Theorem \ref{Scholze} are results for adic spectra. Here the base field $K$ in Theorem \ref{Tate} and Theorem \ref{Berkovich} is a complete valuation field of height at most $1$, and the base field $K$ in Theorem \ref{Scholze} is a perfectoid field. The great appearance of perfectoid theory yields an expectation that there is a good class of Banach algebras (resp.\ affinoid rings) containing both reduced affinoid $K$-algebras and perfectoid $K$-algebras (resp.\ reduced strongly Noetherian Tate rings and perfectoid affinoid $K$-algebras). For example, consider the class of uniform Banach $k$-algebras (resp.\ uniform affinoid rings).

\begin{qst}
\label{question 1}
Is every uniform Banach $k$-algebra (resp.\ uniform affinoid ring) sheafy?
\end{qst}

The sheaf condition for an affinoid $k$-algebra (resp.\ a strongly Noetherian Tate ring) is proved by reducing it to the simplest case where the rational covering is given by the Weierstrass domain and the Laurent domain associated to a common single element. For a general uniform Banach $k$-algebra (resp.\ uniform affinoid ring), the sheaf condition for such a covering is easily verified, and hence if there is a good condition which implies the uniformity and which is preserved under rational localisations, then it gives a good class of uniform Banach algebras which is contained in the class of sheafy Banach algebras. In particular, a natural question arises:

\begin{qst}
\label{question 2}
Is the uniformity preserved under a rational localisation?
\end{qst}

These two questions are not ignorable for us to construct a new comprehensive theory in rigid geometry, and are considered many times by many people before the birth of perfectoid theory. Now they reappeared in the talk by Peter Scholze in ``Hot Topics: Perfectoid Spaces and their Applications'' held in MSRI in Febrary 2014, and we recognised the importance of these questions again. The answer of Question \ref{question 2} is NO both for a Berkovich spectrum and for an adic spectrum as we stated at the beginning. We will give counter-examples in Theorem \ref{not uniform 1} for a Berkovich spectrum, and in Theorem \ref{not uniform 2} and Theorem \ref{not uniform 3} for adic spectra. Furthermore, we verify in Corollary \ref{not sheafy 1} that the counter-example in Theorem \ref{not uniform 1} is not sheafy. Therefore the answer of Question \ref{question 1} is NO for a Berkovich spectrum. We remark that the result of Theorem \ref{not uniform 2} obviously includes the fact that the strong uniformity of an affinoid ring is not preserved under a rational localisation. We also verify in Theorem \ref{not sheafy} that the example in Theorem \ref{not uniform 3} is also an example of a non-sheafy uniform affinoid ring. Thus the answer of Question \ref{question 1} is NO also for an adic spectrum. On the other hand, we also have several affirmative results related to them. First, we introduce the notion of local uniformity in Definition \ref{locally uniform}. The local uniformity implies the uniformity, and is preserved under rational localisations by definition. Therefore the local uniformity is one of desired conditions stronger than the uniformity, and we verify that the local uniformity implies the sheaf condition in Theorem \ref{sheafy}. Second, we prove that every sheafy, locally perfectoid, affinoid ring is a perfectoid affinoid ring if and only if it satisfies a certain vanishing condition of the second cohomology in Theorem \ref{locally perfectoid}. This implies that in the case where the base field is of positive characteristic, a perfectoid space is an affinoid perfectoid space if and only if it is an affinoid space.

\section{Preliminaries}
\label{Preliminaries}

\subsection{Banach Spaces}
\label{Banach Spaces}

For conventions of seminorms on a $k$-vector space, see \cite{Ber} \S 1.1 or \cite{BGR} \S 2. For a seminormed $k$-vector space $(V,\n{\cdot}_{V})$, we set $V(r) \coloneqq \set{v \in V}{\n{v}_{V} \leq r}$ and $V(r-) \coloneqq \set{v \in V}{\n{v}_{V} < r}$ for each $r \in ( 0, \infty )$. We put $O_k \coloneqq k(1)$ and $m_k \coloneqq k(1-)$.

\begin{dfn}
Let $V$ be a $k$-vector space, and $W \subset V$ an $O_k$-submodule. We set $W^{\t{ac}} \coloneqq \bigcap_{r \in (1,\infty)} k(r)W$. We say that $W$ is {\it AC} if $W = W^{\t{ac}}$, and is a {\it lattice} if the multiplication $k \otimes_{O_k} W \to V$ is an isomorphism of $k$-vector spaces.
\end{dfn}

We have $W^{\t{ac} \t{ac}} = W^{\t{ac}}$, and hence $W^{\t{ac}}$ is the smallest AC $O_k$-submodule of $V$ containing $W$. If the valuation of $k$ is discrete, then $W^{\t{ac}} = W$. If $\v{k}$ is dense in $[0,\infty)$, then $W^{\t{ac}}$ coincides with $W_{*}^{\t{a}} \coloneqq \set{v \in V}{m_k v \subset W}$ introduced in \cite{GR} 2.2.2 and 2.2.9.

\begin{dfn}
An $O_k$-module $M$ is said to be {\it AS} if $\bigcap_{r \in (0,1)} k(r)M = \ens{0}$.
\end{dfn}

\begin{prp}
\label{AS}
Every $O_k$-submodule of an AS $O_k$-module is AS.
\end{prp}

\begin{proof}
The proof is straightforward by definition.
\end{proof}

\begin{crl}
An $O_k$-submodule $W$ of a $k$-vector space is AS if and only if so is $W^{\t{ac}}$.
\end{crl}

\begin{prp}
\label{bounded - AS}
Every bounded $O_k$-submodule of a normed $k$-vector space is AS.
\end{prp}

\begin{proof}
We have $\bigcap_{r \in (0,1)} k(r)W \subset \bigcap_{r \in (0,1)} V(rR) = \ens{0}$ for an $R>>0$.
\end{proof}

Let $V$ be a $k$-vector space, and $W \subset V$ a lattice (resp.\ an AS lattice). We denote by $\n{\cdot}_{V,W}$ the seminorm (resp.\ norm) $V \to [0,\infty)$ given by setting $\n{v}_{V,W} \coloneqq \inf \set{r \in (0,\infty)}{v \in k(r)W}$ for each $v \in V$. Then the closed unit ball of $V$ with respect to $\n{\cdot}_{V,W}$ coincides with $W^{\t{ac}}$, and the equality $\n{\cdot}_{V,W^{\t{ac}}} = \n{\cdot}_{V,W}$ holds.

\begin{prp}
\label{norm - unit ball}
Let $V$ be a seminormed $k$-vector space. Then $V(1)$ is an AC lattice. If $\n{\cdot}_{V}$ is a norm, then $V(1)$ is AS, and $\n{\cdot}_{V}$ is equivalent to $\n{\cdot}_{V,V(1)}$. If the closure of $\v{k} \subset [0,\infty)$ contains $\n{V}_{V}$, then $\n{\cdot}_{V}$ coincides with $\n{\cdot}_{V,V(1)}$.
\end{prp}

The proof is straightforward. If $\v{k}$ is dense in $[0,\infty)$, then $V(1)$ determines $\n{\cdot}_{V}$, and there is an anti-order preserving one-to-one correspondence between the set of AC (resp.\ AS AC) lattices of $V$ and the set of seminorms (resp.\ norms) on $V$.

\begin{lmm}
\label{comensurability}
Let $V$ be a normed $k$-vector space. An open lattice $W \subset V$ is bounded if and only if $\n{\cdot}_{V}$ is equivalent to $\n{\cdot}_{V,W}$.
\end{lmm}

\begin{lmm}
\label{completion}
Let $V$ be a seminormed $k$-vector space, and $W \subset V$ an $O_k$-submodule with $W^{\t{ac}} = V(1)$. The closure $\W$ of the image of $W$ in the completion $\V$ of $V$ with respect to $\n{\cdot}_{V}$ is an AS lattice with $\W^{\t{ac}} = \V(1)$. Moreover, the canonical homomorphism $k \otimes_{O_k} \varprojlim_{0 < r < 1} W/k(r)W \to \V$ is an isomorphism of $k$-vector spaces, and its restriction gives an isomorphism $\varprojlim_{0 < r < 1} W/k(r)W \to \W$ of topological $O_k$-modules with respect to the inverse limit topology of the source and the norm topology of the target.
\end{lmm}

Lemma \ref{comensurability} (resp.\ Lemma \ref{completion}) follows from Proposition \ref{norm - unit ball} (resp.\ Proposition \ref{bounded - AS} and Proposition \ref{norm - unit ball}). For a family $(V_i)_{i \in I}$ of Banach $k$-vector spaces indexed by a set $I$, denoting by $\v{V_i}$ the underlying $k$-vector space of $V_i$ for each $i \in I$, we set $\v{\prod_{i \in I} V_i} \coloneqq \set{v = (v_i)_{i \in I} \in \prod_{i \in I} \v{V_i}}{\sup_{i \in I} \n{v_i}_{V_i} < \infty}$, and endow it with the norm $\n{\cdot}_{\prod_{i \in I} V_i} \colon v = (v_i)_{i \in I} \mapsto \sup_{i \in I} \n{v_i}_{V_i}$. Then $\prod_{i \in I} V_i \coloneqq (\v{\prod_{i \in I} V_i}, \n{\cdot}_{\prod_{i \in I} V_i})$ is a Banach $k$-vector space, and satisfies the universality of a direct product of $(V_i)_{i \in I}$ in the category of Banach $k$-vector spaces and submetric $k$-linear homomorphisms.

\subsection{Banach Algebras}
\label{Banach Algebras}

A {\it $k$-algebra} is always assumed to be a commutative associative unital non-zero $k$-algebra. Let $\A$ be a normed $k$-algebra. We denote by $\A^{\circ} \subset \A$ is a subset of power-bounded elements. Then $\A^{\circ}$ is an open $O_k$-subalgebra of $\A$ containing $\A(1)$, and depends only on the equivalence class of the norm of $\A$. In particular, we have $k^{\circ} = O_k$. We emphasise that this convention of $\A^{\circ}$ is not compatible with that in \cite{Ber} \S 2.4.

\begin{lmm}
\label{bounded ac - spectral radius}
Suppose that $\v{k}$ is dense in $[0,\infty)$. For any Banach $k$-algebra $\A$, $(\A^{\circ})^{\t{ac}}$ coincides with the subring consisting of elements $a$ with $\rho_{\A}(a) \leq 1$.
\end{lmm}

\begin{proof}
Let $a \in \A$. Suppose $\rho_{\A}(a) \leq 1$. Let $\epsilon \in m_k$. We have $\rho_{\A}(\epsilon a) = \v{\epsilon} \rho_{\A}(a) < 1$, and hence $\epsilon a$ is topologically nilpotent. Therefore $\epsilon a \in \A^{\circ}$. Thus $a \in (\A^{\circ})^{\t{ac}}$. On the other hand, suppose $a \in (\A^{\circ})^{\t{ac}}$. Assume $\rho_{\A}(a) > 1$. Since $\v{k}$ is dense in $[0,\infty)$, there is an $\epsilon \in m_k$ such that $\rho_{\A}(a)^{-1} < \v{\epsilon} < 1$. However, $\epsilon a \in \A^{\circ}$ and hence $\rho_{\A}(\epsilon a) \leq 1$ by the definition of $\rho_{\A}$. It contradicts the inequality $\rho_{\A}(\epsilon a) = \v{\epsilon} \rho_{\A}(a) > 1$.
\end{proof}

Let $\A$ be a Banach $k$-algebra. We denote by $\M(\A)$ the non-empty compact Hausdorff space consisting of bounded multiplicative seminorms on $\A$ (\cite{Ber} 1.2.1 Theorem). An $x \in \M(\A)$ is said to be {\it $k$-rational} if $\A/\set{f \in \A}{x(f) = 0} \cong k$. We denote by $\M(\A)(k) \subset \M(\A)$ the subset of $k$-rational points. We say that $\A$ is {\it uniform} if its norm is power-multiplicative, and is a {\it Banach function algebra} over $k$ if the spectral radius function $\rho_{\A} \colon \A \to [ 0, \infty ) \colon a \mapsto \inf_{n \in \N} \n{a^n}_{\A}^{1/n}$ is a complete power-multiplicative norm. By Banach's open mapping theorem (\cite{Bou} Theorem I.3.3/1), $\A$ is a Banach function algebra over $k$ if and only if $\n{\cdot}_{\A}$ is equivalent to $\rho_{\A}$ as seminorms.

\begin{lmm}
\label{bounded - spectral radius}
For any Banach function algebra $\A$ over $k$, $\A^{\circ}$ coincides with the AC subring consisting of elements $a$ with $\rho_{\A}(a) \leq 1$. Moreover, $\rho_{\A}$ is a unique complete power-multiplicative norm equivalent to $\n{\cdot}_{\A}$.
\end{lmm}

\begin{proof}
The first assertion follows from Proposition \ref{norm - unit ball}. Since $\A$ is a Banach function algebra over $k$, $\rho_{\A}$ is a complete power-multiplicative seminorm equivalent to $\n{\cdot}_{\A}$. Since  $\rho_{\A}$ depends only on the equivalence class of $\n{\cdot}_{\A}$ by \cite{Ber} 1.3.1 Theorem, a complete power-multiplicative norm on $\A$ equivalent to $\n{\cdot}_{\A}$ is unique.
\end{proof}

\begin{crl}
\label{circ - unit ball}
For any uniform Banach $k$-algebra $\A$, the equality $\A^{\circ} = \A(1)$ holds.
\end{crl}

\begin{prp}
\label{ring - ac}
Let $\A$ be a Banach $k$-algebra. For any $O_k$-subalgebra $A_0$ of $\A$, $A_0^{\t{ac}}$ is an $O_k$-subalgebra of $\A$.
\end{prp}

\begin{proof}
If the valuation of $k$ is discrete, then $A_0^{\t{ac}} = A_0$ and hence is an $O_k$-subalgebra. Suppose that $\v{k}$ is dense in $[0,\infty)$. Let $(a,b) \in A_0^{\t{ac}} \times A_0^{\t{ac}}$. Let $\epsilon \in m_k$. Since $\v{k}$ is dense in $[0,\infty)$, there is an $(\epsilon_1, \epsilon_2) \in m_k \times m_k$ with $\epsilon_1 \epsilon_2 = \epsilon$. We obtain $\epsilon(ab) = (\epsilon_1 a)(\epsilon_2 b) \in A_0$, and hence $ab \in A_0^{\t{ac}}$. Thus $A_0^{\t{ac}}$ is closed under the multiplication.
\end{proof}

\begin{prp}
\label{circ - int cl}
For any normed $k$-algebra $\A$, $\A^{\circ} \subset \A$ is integrally closed.
\end{prp}

\begin{proof}
Let $a \in \A$ be an element integral over $\A^{\circ}$, and $P(T) = \sum_{h = 0}^{n} P_h T^h \in \A^{\circ}[T]$ denote the minimum polynomial of $a$. By $P_0, \ldots, P_h \in \A^{\circ}$, they generate a bounded $O_k$-subalgebra $\A_0 \subset \A$. By the continuity of the multiplication, $M \coloneqq \sum_{h = 0}^{n-1} \A_0 a^h \subset \A$ is a bounded $\A_0$-submodule. Since $a^h P(a) = 0$, we have $a^{h+n} \in M$ by induction on $h$ for any $h \in \N$. Thus $a \in \A^{\circ}$.
\end{proof}

\begin{dfn}
Let $A$ be a $k$-algebra. An $O_k$-subalgebra $A_0 \subset A$ is said to be {\it PS} if the equality $A_0 = \set{a \in A}{a^2 \in A_0}$ holds.
\end{dfn}

\begin{prp}
\label{int cl - ac}
Let $\A$ be a Banach $k$-algebra. For any open PS (resp.\ integrally closed) $O_k$-subalgebra $\A_0 \subset \A$, so is $\A_0^{\t{ac}} \subset \A$.
\end{prp}

\begin{proof}
The subset $\A_0^{\t{ac}} \subset \A$ is an $O_k$-subalgebra by Proposition \ref{ring - ac}, and is open because $\A_0 \subset \A_0^{\t{ac}}$. If the valuation of $k$ is discrete, then $\A_0^{\t{ac}} = \A_0$. Therefore we may assume that $\v{k}$ is dense in $[0,\infty)$. Let $a \in \A$ be an element with $a^2 \in \A_0^{\t{ac}}$ (resp.\ integral over $\A_0^{\t{ac}}$). Put $P \coloneqq T^2 - a^2 \in \A_0^{\t{ac}}[T]$ (resp.\ Take a monic $P \in \A_0^{\t{ac}}[T]$ with $P(a) = 0$). Let $d \in \N \backslash \ens{0}$ be the degree of $P$. For any $\epsilon \in m_k$, we have $\epsilon^d P(\epsilon^{-1} T) \in \A_0[T]$ and hence $\epsilon a$ satisfies $(\epsilon a)^2 \in \A_0$ (resp.\ is integral over $\A_0$). Since $\A_0$ is PS (resp.\ integrally closed), we obtain $\epsilon a \in \A_0$. Thus $a \in \A_0^{\t{ac}}$. We conclude that $\A_0^{\t{ac}}$ is PS (resp.\ integrally closed).
\end{proof}

\begin{lmm}
\label{ring of definition - norm}
Let $\A$ be a $k$-algebra. For any AS lattice $\A_0 \subset \A$ which is a subring, $\n{\cdot}_{\A,\A_0}$ is a norm of $\A$ as a $k$-algebra. Moreover, if $\v{k}$ is dense in $[0,\infty)$ and $\A_0$ is PS, then $\n{\cdot}_{\A,\A_0}$ is power-multiplicative.
\end{lmm}

\begin{proof}
The first assertion is obvious. Put $\n{\cdot}_0 \coloneqq \n{\cdot}_{\A,\A_0}$. Suppose that $\A_0$ is PS. Then $\A_0^{\t{ac}}$ is PS by Proposition \ref{int cl - ac}. Let $a \in \A$. In order to prove the power-multiplicativity of $\n{\cdot}_0$, it suffices to verify $\n{a}^2 \leq \n{a^2}$ because the converse inequality follows from the submultiplicativity of $\n{\cdot}_0$. If $a = 0$, then $\n{a^2}_0 = \n{a}_0^2 = 0$. Suppose $a \neq 0$. For any $r \in (\n{a^2}_0,\infty)$, there is a $c \in k^{\times}$ such that $\n{a^2}_0 \leq \v{c}^2 < r$ because $\v{k}$ is dense in $[0,\infty)$. Then $\n{(c^{-1}a)^2}_0 = \v{c}^{-2} \n{a^2}_0 \leq 1$, and hence $(c^{-1}a)^2 \in \A_0^{\t{ac}}$. Since $\A_0^{\t{ac}}$ is PS, we obtain $c^{-1}a \in \A_0^{\t{ac}}$ and hence $\n{c^{-1}a}_0 \leq 1$. It implies $\n{a}_0^2 \leq \v{c}^2 < r$. Thus $\n{a^2}_0 \leq \n{a}_0^2$. We conclude that $\n{\cdot}_0$ is power-multiplicative.
\end{proof}

\begin{prp}
\label{uniform Banach}
A Banach $k$-algebra $\A$ is a Banach function algebra over $k$ if and only if $\A^{\circ}$ is bounded. Moreover, in the case where $\v{k}$ is dense in $[0,\infty)$, $\A$ is a uniform Banach $k$-algebra if and only if $\A(1) = \A^{\circ}$.
\end{prp}

\begin{proof}
The necessity implications of the two assertions follow from Lemma \ref{bounded - spectral radius}, Corollary \ref{circ - unit ball}, Proposition \ref{norm - unit ball} for $\rho_{\A}$ and $\A^{\circ}$, and Lemma \ref{comensurability} for $\n{\cdot}_{\A}$ and $\A^{\circ}$. Suppose that $\A^{\circ}$ is bounded. We prove $(\A^{\circ})^{\t{ac}} = \A^{\circ}$. Let $a \in (\A^{\circ})^{\t{ac}}$. We have $a^n \in (\A^{\circ})^{\t{ac}}$ for any $n \in \N$ by Proposition \ref{ring - ac}. Take a $c \in k^{\times}$ such that $\A^{\circ} \subset c \A(1)$. Since $(c \A(1))^{\t{ac}} = c \A(1)^{\t{ac}} = c \A(1)$, we have $(\A^{\circ})^{\t{ac}} \subset (c \A(1))^{\t{ac}} = c \A(1)$. It implies $\set{a^n}{n \in \N} \subset c \A(1)$, and hence $a \in \A^{\circ}$. We obtain $(\A^{\circ})^{\t{ac}} = \A^{\circ}$. Therefore $\A^{\circ}$ coincides with the closed unit ball with respect to $\rho_{\A}$ by Lemma \ref{bounded ac - spectral radius}, and hence $\rho_{\A}$ is a norm because $\A^{\circ}$ is AS by Proposition \ref{bounded - AS}. Since $\A^{\circ}$ is an open bounded subspace of $\A$, $\rho_{\A}$ is equivalent to $\n{\cdot}_{\A}$ by Proposition \ref{norm - unit ball} and Lemma \ref{comensurability}. We conclude that $\A$ is a Banach function algebra over $k$. Suppose that $\v{k}$ is dense in $[0,\infty)$ and $\A(1) = \A^{\circ}$. Since $\A^{\circ}$ is PS by Proposition \ref{circ - int cl}, $\n{\cdot}_{\A}$ is power-multiplicative by Proposition \ref{norm - unit ball} and Lemma \ref{ring of definition - norm}. Therefore $\A$ is uniform.
\end{proof}

\begin{crl}
\label{uniformisation}
Suppose that $\v{k}$ is dense in $[0,\infty)$. For any Banach function algebra $\A$ over $k$, $\n{\cdot}_{\A,\A^{\circ}}$ is a unique power-multiplicative norm of $\A$ equivalent to $\n{\cdot}_{\A}$.
\end{crl}

Let $(\A_i)_{i \in I}$ be a family of Banach $k$-algebras. Then $\prod_{i \in I} \A_i$ is a Banach $k$-algebra, and satisfies the universality of a direct product of $(\A_i)_{i \in I}$ in the category of Banach $k$-algebras and submetric $k$-algebra homomorphisms. Moreover, $\prod_{i \in I} \A_i$ is uniform if and only if so is $\A_i$ for any $i \in I$. We remark that a direct product does not preserve the class of Banach function algebras over $k$, but at least it is true that if $\prod_{i \in I} \A_i$ is a Banach function algebra over $k$, then so is $\A_i$ because the zero-extension $\A_i \hookrightarrow \prod_{i \in I} \A_i$ is an isometric multiplicative $k$-linear homomorphism for any $i \in I$.

\section{Norms on Affinoid Rings}
\label{Norms on Affinoid Rings}

In this paper, an affinoid ring $\A = (\A^{\triangleright},\A^{+})$ is assumed to satisfy that $\A^{\triangleright}$ is a complete Tate ring. We study the relation between properties of an affinoid ring and norms on the underlying topological ring in \S \ref{Affinoid Rings}. We compare the notion of a rational localisation of the Berkovich spectrum of a Banach $k$-algebra with that of the adic spectrum of an affinoid ring in \S \ref{Rational Localisations}.

\subsection{Affinoid Rings}
\label{Affinoid Rings}

For conventions of an affinoid ring, see \cite{Hub1}, \cite{Hub2}, or \cite{Hub3}. We always assume that the underlying f-adic ring of an affinoid ring is a complete Tate ring. Following the convention for a normed $k$-vector space, for an f-adic ring $A$, we denote by $A^{\circ} \subset A$ the subring of power-bounded elements. We remark that similar proof of Proposition \ref{circ - int cl} works for an f-adic ring $A$, and hence $A^{\circ} \subset A$ is integrally closed. For a topological field $K$, an {\it affinoid ring over $K$} is an affinoid ring $\A$ such that $\A^{\triangleright}$ is an f-adic ring over $K$.

\begin{dfn}
\label{norm - affinoid}
Let $\A$ be a Banach $k$-algebra. We denote by $\A^{\t{ad}+}$ the integral closure of $\A(1)$ in $\A$, and call $\A^{\t{ad}} \coloneqq (\A,\A^{\t{ad}+})$ {\it the affinoid algebra associated to $\A$}.
\end{dfn}

For a Banach $k$-algebra $\A$, we have $\A(1) \subset \A^{\circ}$ and hence $\A^{\t{ad}+}$ is contained in $\A^{\circ}$ by Proposition \ref{circ - int cl}, and $\A^{\t{ad}}$ an affinoid ring over $k^{\t{ad}}$. If $\A$ is a uniform Banach $k$-algebra, then $\A(1) = \A^{\circ}$ is integrally closed and hence $\A^{\t{ad}} = (\A,\A(1)) = (\A,\A^{\circ})$. Let $A$ be an f-adic ring over $k$. An {\it $O_k$-subalgebra of definition of $A$} is a ring of definition of $A$ which is an $O_k$-subalgebra.

\begin{prp}
\label{Banach - affinoid}
Every f-adic ring $A$ over $k$ admits an $O_k$-subalgebra of definition, and for any $O_k$-subalgebra $A_0 \subset A$ of definition, $\n{\cdot}_{A,A_0}$ is a complete norm of $A$ as a $k$-algebra which gives its original topology.
\end{prp}

\begin{proof}
The $O_k$-subalgebra generated by a ring of definition with an ideal $I$ of definition is an $O_k$-subalgebra of definition with an ideal of definition generated by $I$. Let $A_0 \subset A$ be an $O_k$-subalgebra of definition. We have $\epsilon A_0^{\t{ac}} \subset A_0 \subset A_0^{\t{ac}}$ for any $\epsilon \in k(1-)$, and hence $A_0^{\t{ac}}$ is also an $O_k$-subalgebra of definition. The topology of $A$ is given by $\n{\cdot}_{A,A_0}$ because $A_0^{\t{ac}}$ is the closed unit ball of $A$ with respect to $\n{\cdot}_{A,A_0}$, and $\n{\cdot}_{A,A_0}$ is a norm of $A$ as a $k$-algebra by Lemma \ref{ring of definition - norm}.
\end{proof}

For an f-adic ring $A$ over $k$, we define a seminorm $\rho_{A} \colon A \to [0,\infty)$ as the spectral radius function on the Banach $k$-algebra which shares the underlying $k$-algebra with $A$ and is endowed with a complete norm $\n{\cdot}_{\A}$ in Proposition \ref{Banach - affinoid}. It is independent of the choice of $\n{\cdot}_{\A}$. For an affinoid ring $\A$ over $k$, we put $\rho_{\A} \coloneqq \rho_{\A^{\triangleright}}$.

\begin{prp}
\label{ac int - ac circ}
Suppose that $\v{k}$ is dense in $[0,\infty)$. For any f-adic ring $A$ over $k$ with a ring $A^{+} \subset A$ of integral elements which is an $O_k$-subalgebra, the equality $(A^{+})^{\t{ac}} = (A^{\circ})^{\t{ac}}$ holds.
\end{prp}

\begin{proof}
The assertion immediately follows from Lemma \ref{bounded ac - spectral radius} and \cite{Ber} 1.3.1 Theorem.
\end{proof}

For an affinoid ring $\A$, we denote by $\t{Spa}(\A)$ the small set of (equivalence classes of) continuous valuations $x$ on $\A^{\triangleright}$ with $x(f) \leq 1$ for any $f \in \A^+$ (\cite{Hub1} \S 3 p.\ 467 Definition (iii)). An $x \in \t{Spa}(\A)$ is said to be {\it of height $1$} if a valuation of height $1$ represents $x$. We denote by $\t{Spa}(\A)_1 \subset \t{Spa}(\A)$ the subset of continuous valuations of height $1$. If $\A$ is an affinoid ring over $k$, then every $x \in \t{Spa}(\A)_1$ admits a unique representative $\v{\cdot}_x \colon \A^{\triangleright} \to [0,\infty)$ whose restriction on $k$ coincides with the valuation on $k$. For $x,y \in \t{Spa}(\A)_1$, the equality $\v{\cdot}_x = \v{\cdot}_y$ holds if and only if $x = y$.

\begin{prp}
\label{height 1}
Let $\A$ be an affinoid ring $\A$ over $k$, and $\n{\cdot}_0$ is a complete norm of $\A^{\triangleright}$ in Proposition \ref{Banach - affinoid}. The correspondence $x \mapsto \v{\cdot}_x$ gives a bijective map $\t{Spa}(\A)_1 \to \M(\A^{\triangleright},\n{\cdot}_0)$, and the equality $\rho_{\A}(f) = \sup_{x \in \t{Spa}(\A)_1} \v{f}_x$ holds for any $f \in \A$. In particular, $\t{Spa}(\A)_1$ depends only on $\A^{\triangleright}$.
\end{prp}

\begin{proof}
The second assertion follows from the first assertion and \cite{Ber} 1.3.1 Theorem by the definition of $\rho_{\A}$. In order to verify the first assertion, it suffices to show that a map $x \colon \A^{\triangleright} \to [0,\infty)$ is a continuous valuation with $x(\A^{+}) \subset [0,1]$ and $x |_k = \v{\cdot}$ if and only if $x \in \M(\A^{\triangleright}, \n{\cdot}_0)$. The sufficient implication immediately follows from Lemma \ref{bounded ac - spectral radius} and \cite{Ber} 1.3.1 Theorem. Suppose that $x$ is a continuous valuation with $x(\A^{+}) \subset [0,1]$ and $x |_k = \v{\cdot}$. Then by the continuity of $x$, there is an $r \in (0,\infty)$ such that $\set{f \in \A^{\triangleright}}{\n{f}_0 \leq r}$ is contained in $\set{f \in \A^{\triangleright}}{x(f) \leq 1}$. Assume that there is an $f \in \A^{\triangleright}$ such that $0 \neq \n{f}_0 < x(f)$. Since the valuation of $k$ is not trivial, there is a $\varpi \in k^{\times}$ with $\v{\varpi} < 1$. Take an $N \in \N$ with $r^{-1} \v{\varpi}^{-1} \n{f}_0^N < x(f)^N$. Let $n \in \N$ denote the least number with $r^{-1} \v{\varpi}^n \n{f}_0^N \leq 1$. We have
$x(\varpi^n f^N) = \v{\varpi}^n x(f)^N \geq r^{-1} \v{\varpi}^{n-1} \n{f}_0^N > 1$, and it contradicts the inequality $\n{\varpi^n f^N}_0 = \v{\varpi}^n \n{f^N}_0 \leq r$. Therefore we obtain $x(f) \leq \n{f}_0$ for any $f \in \A^{\triangleright}$. Thus $x \in \M(\A^{\triangleright}, \n{\cdot}_0)$.
\end{proof}

\begin{dfn}
An f-adic ring $A$ over $k$ is said to be {\it uniform} if $A^{\circ}$ is bounded. An affinoid ring $\A$ over $k$ is said to be {\it uniform} if $\A^{\triangleright}$ is uniform, and is said to be {\it strongly uniform} if $\A$ is uniform and $\A^{+} = (\A^{\triangleright})^{\circ}$.
\end{dfn}

For any uniform affinoid ring $\A$ over $k$, $\A^{+}$ is a ring of definition. Therefore every strongly uniform affinoid ring over $k^{\t{ad}}$ is isomorphic to $\A^{\t{ad}}$ for a Banach $k$-algebra.

\begin{prp}
\label{almost uniform - affinoid}
For an f-adic ring $A$ over $k$, the following are equivalent:
\begin{itemize}
\item[(i)] The f-adic ring $A$ over $k$ is uniform.
\item[(ii)] The f-adic ring $A$ over $k$ admits a norm which gives its original topology and with respect to which $A$ is a Banach function algebra over $k$.
\item[(iii)] The f-adic ring $A$ over $k$ admits a unique complete power-multiplicative norm as a $k$-algebra which gives its original topology.
\end{itemize}
In addition if $\v{k}$ is dense in $[0,\infty)$, then they are equivalent to the following:
\begin{itemize}
\item[(iv)] The f-adic ring $A$ over $k$ admits a PS ring of definition which is an $O_k$-subalgebra, and for any PS ring $A_0 \subset A$ of definition which is an $O_k$-subalgebra, $\n{\cdot}_{A,A_0}$ is a complete power-multiplicative norm of $A$ which gives its original topology.
\item[(v)] The open $O_k$-subalgebra $A^{\circ} \subset A$ is an AC PS ring of definition.
\end{itemize}
\end{prp}

\begin{proof}
The first assertion follows from Lemma \ref{bounded - spectral radius}, Proposition \ref{uniform Banach}, and Proposition \ref{Banach - affinoid}. Suppose that $\v{k}$ is dense in $[0,\infty)$. The condition (iv) implies (v) by Lemma \ref{bounded - spectral radius} and Proposition \ref{circ - int cl}. The condition (v) implies (ii) by Proposition \ref{circ - int cl} and Lemma \ref{ring of definition - norm}. We verify that the condition (i) implies (iv). The $O_k$-subalgebra $A^{\circ} \subset A$ is integrally closed and bounded by Proposition \ref{circ - int cl}, and hence is a PS ring of definition. Take a PS ring $A_0 \subset A$ of definition. which is an $O_k$-subalgebra. Then $A_0^{\t{ac}}$ is also an $O_k$-subalgebra of definition by the proof of Proposition \ref{Banach - affinoid}, and is PS by Proposition \ref{int cl - ac}. Therefore the topology of $A$ is given by $\n{\cdot}_{A,A_0}$ because $A_0^{\t{ac}}$ is the closed unit ball of $A$ with respect to $\n{\cdot}_{A,A_0}$, and $\n{\cdot}_{A,A_0}$ is a power-multiplicative norm of $A$ as a $k$-algebra by Lemma \ref{ring of definition - norm}. Thus $A$ satisfies (iv).
\end{proof}

\begin{crl}
\label{spectral radius}
For any uniform affinoid ring $\A$ over $k$, $\rho_{\A}$ is a complete power-multiplicative norm on $\A^{\triangleright}$ which gives its original topology and for which the closed unit ball of $\A^{\triangleright}$ coincides with $(\A^{\triangleright})^{\circ}$, and is unique among such seminorms. If $\v{k}$ is dense in $[0,\infty)$, then the equality $\rho_{\A} = \n{\cdot}_{\A^{\triangleright},(\A^{\triangleright})^{\circ}}$ holds.
\end{crl}

\begin{crl}
\label{circ - ac}
For any uniform affinoid ring $\A$ over $k$, $(\A^{\triangleright})^{\circ}$ is AC in $\A^{\triangleright}$.
\end{crl}

\begin{crl}
\label{su - ac}
Suppose that $\v{k}$ is dense in $[0,\infty)$. A uniform affinoid ring $\A$ over $k$ is strongly uniform if and only if $\A^{+}$ is AC in $\A^{\triangleright}$.
\end{crl}

\begin{prp}
\label{uniform - affinoid}
An affinoid ring $\A$ over $k$ is strongly uniform if and only if $\A$ is isomorphic to the affinoid algebra associated to a uniform Banach $k$-algebra.
\end{prp}

\begin{proof}
The assertion immediately follows from Proposition \ref{circ - int cl} and Corollary \ref{spectral radius}.
\end{proof}

\begin{prp}
\label{uniformisation 2}
Suppose that $\v{k}$ is dense in $[0,\infty)$. An affinoid ring $\A$ over $k^{\t{ad}}$ is uniform if and only if $(\A^{\triangleright})^{\circ} \subset \A^{\triangleright}$ is AC and $(\A^{\triangleright},(\A^{+})^{\t{ac}})$ is strongly uniform.
\end{prp}

\begin{proof}
The sufficient implication follows from Proposition \ref{int cl - ac}, and the necessary implication follows from Corollary \ref{circ - ac} and Corollary \ref{su - ac}.
\end{proof}

\subsection{Rational Localisations}
\label{Rational Localisations}

Let $n \in \N \backslash \ens{0}$ and $r = (r_1,\ldots,r_n) \in (0,\infty)^n$. We put $(1,\ldots,1) \coloneqq 1_n \in (0,\infty)^n$. For each $h = (h_1,\ldots,h_n) \in \N^n$, we set $\v{h} \coloneqq h_1 + \cdots + h_n$, $r^h \coloneqq r_1^{h_1} \cdots r_n^{h_n}$, and $\mathbb{T}^h \coloneqq T_1^{h_1} \cdots T_n^{h_n} \in \Z[T_1,\ldots,T_h]$. Let $\A$ be a Banach $k$-algebra. We denote by $\A \ens{r^{-1}\mathbb{T}}$ the $k$-subalgebra of $\A[[T_1,\ldots,T_n]]$ consisting of elements $f = \sum_{h \in \N^n} f_h \mathbb{T}^h$ with $\n{f_h} r^h \xrightarrow[]{\v{h} \to \infty} 0$, and endow it with the complete norm $\n{\cdot}_{\A \ens{r^{-1}\mathbb{T}}} \colon \A \ens{r^{-1}\mathbb{T}} \to [0,\infty) \colon \sum_{h \in \N^n} f_h \mathbb{T}^h \mapsto \sup_{h \in \N^n} \n{f_h} r^h$. We call $\n{\cdot}_{\A \ens{r^{-1}\mathbb{T}}}$ {\it the Gauss norm of radius $r$}, and the Banach $k$-algebra $(\A \ens{r^{-1}\mathbb{T}}, \n{\cdot}_{\A \ens{r^{-1}\mathbb{T}}})$ {\it the Tate algebra of radius $r$ over $\A$}. In particular, we put $\A \ens{T_1,\ldots,T_n} \coloneqq \A \ens{1_n^{-1}\mathbb{T}}$. For any Banach $k$-algebra $\B$ and any bounded $k$-algebra homomorphism $\varphi \colon \A \to \B$ satisfying $\varphi(f_0) \in \B^{\times}$ and $\n{\varphi(f_0)^{-1} \varphi(f_i)}_{\B} \leq r_i$ for any $i \in \N \cap [1,n]$, there is a unique bounded homomorphism $\widetilde{\varphi} \colon \A \ens{r^{-1}\mathbb{T}} \to \B$ such that $\widetilde{\varphi}(T_i) = \varphi(f_0)^{-1} \varphi(f_i)$ for any $i \in \N \cap [1,n]$ and the composite of the canonical embedding $\A \hookrightarrow \A \ens{r^{-1}\mathbb{T}}$ and $\widetilde{\varphi}$ coincides with $\varphi$. We call this extension property {\it the extension property of the Tate algebra}. For each $f_0 \in \A$ and $f = (f_1,\ldots,f_n) \in \A^n$ with $\sum_{i = 1}^{n} \A f_i = \A$, we denote by $\A \ens{r^{-1}f_0^{-1}f}$ the quotient of $\A \ens{r^{-1}\mathbb{T}}$ by the closure of the ideal generated by $\set{f_0 T_i - f_i}{i \in \N \cap [1,n]}$. In particular, we put $\A \ens{f_0^{-1}f} \coloneqq \A \ens{1_n^{-1}f_0^{-1}f}$, $\A \ens{r^{-1}f} \coloneqq \A \ens{r^{-1}1^{-1}f}$, $\A \ens{f} \coloneqq \A \ens{1_n^{-1}f}$, and so on. We denote by $\M(\A) \ens{r^{-1}f_0^{-1}f} \subset \M(\A)$ the closed subset of points $x$ with $\v{f_i(x)} \leq r_i \v{f_0(x)}$ for any $i \in \N \cap [1,n]$. A subset of $\M(\A)$ of such a form is called a {\it rational domain}. 

\begin{prp}
Let $\A$ be a Banach $k$-algebra. For each $f_0 \in \A$ and $f = (f_1,\ldots,f_n) \in \A^n$ with $\sum_{i = 1}^{n} \A f_i = \A$, the map $\M(\A \ens{r^{-1}f_0^{-1}f}) \to \M(\A)$ associated to the canonical homomorphism $\A \to \A \ens{r^{-1}f_0^{-1}f}$ is a homeomorphism onto $\M(\A) \ens{r^{-1}f_0^{-1}f}$.
\end{prp}

\begin{proof}
Let $\varphi \colon \M(\A \ens{r^{-1}f_0^{-1}f}) \to \M(\A)$ denote the given continuous map, which is injective because $f_0^{-1}f_1, \ldots, f_0^{-1}f_n$ generates a dense $\A$-subalgebra of $\A \ens{r^{-1}f_0^{-1}f}$. Since Berkovich spectra are compact and Hausdorff, $\varphi$ is a homeomorphism onto the closed image. Therefore it suffices to verify that the image of $\varphi$ coincides with $\M(\A) \ens{r^{-1}f_0^{-1}f}$. Let $x \in \M(\A \ens{r^{-1}f_0^{-1}f})$. For any $i \in \N \cap [1,n]$, we have $\v{f_i(\varphi(x))} = \v{f_i(x)} = \v{f_0(x)} \ \v{(f_0^{-1}f_i)(x)} = \v{f_0(\varphi(x))} \ \v{(f_0^{-1}f_i)(x)} \leq \v{f_0(\varphi(x))} \n{f_0^{-1}f_i}_{\A \ens{r^{-1}f_0^{-1}f}} \leq r_i \v{f_0(\varphi(x))}$. Therefore $\varphi(x) \in \M(\A) \ens{r^{-1}f_0^{-1}f}$.

\vspace{0.2in}
Let $y \in \M(\A) \ens{r^{-1}f_0^{-1}f}$, and $\eta \colon \A \to \kappa(y)$ be the canonical bounded homomorphism to the extension $\kappa(y)/k$ of valuation fields corresponding to $y$ by \cite{Ber} 1.2.2 Remark. Since $\v{f_i(y)} \leq r_i \v{f_0(y)}$ for any $i \in \N \cap [1,n]$, it extends to a bounded homomorphism $\A \ens{r^{-1}\mathbb{T}} \to \kappa(y)$ sending $T_i$ to $\eta(f_0)^{-1} \eta(f_i)$ for each $i \in \N \cap [1,n]$ by the extension property of the Tate algebra. The kernel obviously contains $f_0T_i-f_i$ for each $i \in \N \cap [1,n]$, and hence it induces a bounded $k$-algebra homomorphism $\A \ens{r^{-1}f_0^{-1}f} \to \kappa(y)$. It corresponds to an $x \in \M(\A \ens{r^{-1}f_0^{-1}f})$. Since the restriction of $x$ on the image of $\A$ coincides with $y$, we obtain $y = \varphi(x)$. Thus the image of $\varphi$ coincides with $\M(\A) \ens{r^{-1}f_0^{-1}f}$.
\end{proof}

\begin{lmm}
\label{Nullstellensatz}
Let $\A$ be a Banach $k$-algebra. For an $a \in \A$, $\A \ens{a} = 0$ if and only if $a \in \A^{\times}$ and $\v{a^{-1}(x)} < 1$ for any $x \in \M(\A)$.
\end{lmm}

\begin{proof}
The assertion follows from \cite{Ber} 1.2.1 Theorem and \cite{Ber} 1.2.4 Corollary.
\end{proof}

The underlying $k$-algebra of the Banach $k$-algebra $k \ens{r^{-1}\mathbb{T}}$ is Noetherian, and every ideal of it is closed by \cite{Ber} 2.1.3 Proposition. If $\A$ is isomorphic to the quotient of such a Banach $k$-algebra, the isomorphism class of $\A \ens{f_0^{-1}f}$ for an $f_0 \in \A$ and $f = (f_1,\ldots,f_n) \in \A^n$ with $\sum_{i = 1}^{n} \A f_i = \A$ depends only on the rational domain $\M(\A) \ens{f_0^{-1}f}$, and hence is denoted by $\Opn_{\M(\A)}(\M(\A) \ens{f_0^{-1}f})$. The correspondence $U \rightsquigarrow \Opn_{\M(\A)}(U)$ gives a presheaf of complete topological $k$-algebras on the Grothendieck topology generated by rational subsets of $\M(\A)$. We remark that there is no canonical norm of $\Opn_{\M(\A)}(U)$ when $\A$ is not reduced, and hence we have to forget the norm on each section. For more details of rational localisations of such $\A$, see \cite{Ber} \S 2.2. On the other hand, if $\A$ is not isomorphic to a quotient of the Banach $k$-algebra of the form $k \ens{r^{-1}\mathbb{T}}$, the isomorphism class of $\A \ens{f_0^{-1}f}$ for an $f_0 \in \A$ and an $f = (f_1,\ldots,f_n) \in \A^n$ with $\sum_{i = 1}^{n} \A f_i = \A$ depends on the presentation $(n,(f_0,f))$ of $\M(\A) \ens{f_0^{-1}f}$ in general, and hence the presheaf $\Opn_{\M(\A)}$ is not well-defined.

\vspace{0.2in}
We recall the relation between rational localisations of a Berkovich spectrum and an adic spectrum. Let $f_0 \in \A$ and $f =(f_1,\ldots,f_n) \in \A^n$ with $\sum_{i = 1}^{n} \A f_i = \A$. We define $\A(1) \langle f_0^{-1}f \rangle \coloneqq \varprojlim_{0 < r < 1} \A(1)[f_0^{-1}f_1, \ldots, f_0^{-1}f_n]/\A(r)[f_0^{-1}f_1, \ldots, f_0^{-1}f_n]$ and $\A \langle f_0^{-1}f \rangle \coloneqq k \otimes_{O_k} (\A(1) \langle f_0^{-1}f \rangle)$, where $\A(r)[f_0^{-1}f_1, \ldots, f_0^{-1}f_n] \subset \A[f_0^{-1}]$ denotes the smallest $O_k$-submodule closed under the addition and the multiplication containing the image of $\A(r)$ and $\ens{f_0^{-1}f_1, \ldots, f_0^{-1}f_n}$ for each $r \in (0,1)$. We have $\A(1) \langle f_0^{-1}f \rangle \cong \varprojlim_{0 < r < 1} \A(1)[f_0^{-1}f_1, \ldots, f_0^{-1}f_n]/k(r)\A(1)[f_0^{-1}f_1, \ldots, f_0^{-1}f_n]$ as topological $O_k$-modules by Lemma \ref{completion} for the seminorm on $\A[f_0^{-1}] = \A[f_0^{-1}f_1,\ldots,f_0^{-1}f_n]$ induced by the Gauss norm of radius $1_n$. Therefore its underlying $O_k$-module is torsion-free and hence is flat. It implies that the canonical homomorphism $\A(1) \langle f_0^{-1}f \rangle \to \A \langle f_0^{-1}f \rangle$ is injective, and we identify $\A(1) \langle f_0^{-1}f \rangle$ with its image. We equip $\A \langle f_0^{-1}f \rangle$ with the complete norm $\n{\cdot}_{\A \langle f_0^{-1}f \rangle} \coloneqq \n{\cdot}_{\A \langle f_0^{-1}f \rangle, \A(1) \langle f_0^{-1}f \rangle}$ as a $k$-algebra (Lemma \ref{ring of definition - norm}).

\begin{prp}
\label{unit ball}
Let $\A$ be a Banach $k$-algebra. For any $f_0 \in \A$ and $f =(f_1, \ldots, f_n) \in \A^n$ with $\sum_{i = 1}^{n} \A f_i = \A$, the equality $\A \langle f_0^{-1}f \rangle (1) = \A(1) \langle f_0^{-1}f \rangle^{\t{ac}}$ holds.
\end{prp}

\begin{proof}
The assertion follows from Lemma \ref{completion} for the induced seminorm on $\A[f_0^{-1}]$.
\end{proof}

\begin{prp}
\label{presentation}
Let $\A$ be a Banach $k$-algebra. For any $f_0 \in \A$ and $f = (f_1,\ldots,f_n) \in \A^n$ with $\sum_{i = 1}^{n} \A f_i = \A$, the natural homomorphism $\A \ens{T_1, \ldots, T_n} \to \A \langle f_0^{-1}f \rangle$ sending $T_i$ to the image of $f_0^{-1}f_i$ for each $i \in \N \cap [1,n]$ is an admissible surjective homomorphism, and its kernel coincides with the closure $(f_0T_1 - f_1, \ldots, f_0T_n - f_n)\hat{}$ of the ideal $(f_0T_1 - f_1, \ldots, f_0T_n - f_n) \subset \A \ens{\mathbb{T}}$ generated by $\ens{f_0T_1 - f_1, \ldots, f_0T_n - f_n}$. Moreover, the induced homomorphism $\A \ens{f_0^{-1}f} \to \A \langle f_0^{-1}f \rangle$ is an isometric isomorphism.
\end{prp}

\begin{proof}
Let $\iota \colon \A \to \A \langle f_0^{-1}f \rangle$ denote the canonical bounded homomorphism. By $\sum_{i = 1}^{n} \A f_i = \A$, there is a $(g_1,\ldots,g_n) \in \A^n$ with $\sum_{i = 1}^{n} g_i f_i = 1$. The equality $\iota(f_0) \sum_{i = 1}^{n} \iota(g_i) \iota(f_0^{-1}f_i) = \sum_{i = 1}^{n} \iota(g_i f_i) = \iota(1) = 1$ guarantees that $\iota(f_0)$ is invertible in $\A \langle f_0^{-1}f \rangle$ by \cite{Ber} Corollary 1.2.4. Therefore $\iota$ extends to a unique bounded homomorphism $\widetilde{\iota} \colon \A \ens{T_1, \ldots, T_n} \to \A \langle f_0^{-1}f \rangle$ sending $T_i$ to $\iota(f_0)^{-1} \iota(f_i)$ for each $i \in \N \cap [1,n]$ because $\n{\iota(f_0)^{-1} \iota(f_i)}_{\A \langle f_0^{-1}f \rangle} \leq 1$ for each $i \in \N \cap [1,n]$. For any $g \in \A(1) \langle f_0^{-1}f \rangle$, there is a sequence $(g_j)_{j \in \N} \in \prod_{j = 0}^{\infty} \A(2^{-j})[f_0^{-1}f_1, \ldots, f_0^{-1}f_n]$ such that $g = \sum_{j = 0}^{\infty} \iota(g_j)$ by the definition of the topology of $\A \langle f_0^{-1}f \rangle$. Take a lift $G_j \in \A[T_1, \ldots, T_n]$ of $g_j$ whose coefficients lie in $\A(2^{-j})$ for each $j \in \N$. By the completeness of $\A$, $G \coloneqq \sum_{j = 0}^{\infty} G_j$ converges in $\A \ens{T_1, \ldots, T_n}$. By the continuity of $\widetilde{\iota}$, we have $\widetilde{\iota}(G) = \sum_{j = 0}^{\infty} \widetilde{\iota}(G_j) = \sum_{j = 0}^{\infty} \iota(g_j) = g$. Therefore $\A(1) \langle f_0^{-1}f \rangle \subset \widetilde{\iota}(\A \ens{T_1, \ldots, T_n}(1))$, and $\widetilde{\iota}$ is surjective because $\A \langle f_0^{-1}f \rangle (1) = (\A(1) \langle f_0^{-1}f \rangle)^{\t{ac}}$. Since $\widetilde{\iota}(\A(1)[T_1, \ldots, T_n]) = \iota(\A(1)[f_0^{-1}f_1, \ldots, f_0^{-1}f_n])$ is dense in $\A(1) \langle f_0^{-1}f \rangle$, the image of $\A \ens{T_1, \ldots, T_n}(1)$ by $\widetilde{\iota}$ coincides with $\A(1) \langle f_0^{-1}f \rangle$ by Banach's open mapping theorem (\cite{Bou} Theorem I.3.3/1). The equality $\widetilde{\iota}(f_0T_i - f_i) = \iota(f_0) \widetilde{\iota}(T_i) - \iota(f_i) = 0$ guarantees that $\ker \widetilde{\iota}$ contains $(f_0T_1 - f_1, \ldots, f_0T_n - f_n)\hat{}$. Moreover, $\widetilde{\iota}$ is submetric by $\widetilde{\iota}(\A(1)[T_1, \ldots, T_n]) = \iota(\A(1))[\iota(f_0)^{-1}\iota(f_1), \ldots, \iota(f_0)^{-1}\iota(f_n)] \subset \A(1) \langle f_0^{-1}f \rangle \subset \A \langle f_0^{-1}f \rangle (1)$. Therefore $\widetilde{\iota}$ induces a surjective submetric homomorphism $\iota \ens{f_0^{-1}f} \colon \A \ens{f_0^{-1}f} \to \A \langle f_0^{-1}f \rangle$.

\vspace{0.2in}
We prove that $\iota \ens{f_0^{-1}f}$ is isometric. Denote by $\iota[f_0^{-1}] \colon \A[f_0^{-1}] \to \A \langle f_0^{-1}f \rangle$ the composite of the canonical homomorphism $\A[f_0^{-1}] \to \A \ens{f_0^{-1}f}$ and $\iota \ens{f_0^{-1}f}$. Let $F = \sum_{h \in \N^n} F_h \mathbb{T}^h \in \A \ens{T_1, \ldots, T_n}$. For any $r > \n{\widetilde{\iota}(F)}_{\A \langle f_0^{-1}f \rangle}$, we have $\widetilde{\iota}(F) \in k(r) \A(1) \langle f_0^{-1}f \rangle$ by the definition of $\n{\cdot}_{\A \langle f_0^{-1}f \rangle}$. Take an $N \in \N$ such that $\n{F_h}_{\A} < r$ for any $h \in \N^n$ with $\v{h} > N$. Set $G \coloneqq \sum_{h \in \N^n \cap [1,N]^n} F_h \mathbb{T}^h \in \A[T_1, \ldots, T_n]$. Since $\widetilde{\iota}$ is submetric, we have $\widetilde{\iota}(F - G) \in \A \langle f_0^{-1}f \rangle (r-) \subset k(r) \A(1) \langle f_0^{-1}f \rangle$ and $\iota[f_0^{-1}](G(f_0^{-1}f_1, \ldots, f_0^{-1}f_n)) = \widetilde{\iota}(G) = \widetilde{\iota}(F) - \widetilde{\iota}(F - G) \in k(r) \A(1) \langle f_0^{-1}f \rangle$. Therefore we obtain $G(f_0^{-1}f_1, \ldots, f_0^{-1}f_n) \in \A(r)[f_0^{-1}f_1, \ldots, f_0^{-1}f_n]$ by the definition of $\iota$, and there is an $H \in \A[T_1, \ldots, T_n]$ such that $H(f_0^{-1}f_1, \ldots, f_0^{-1}f_n) = G(f_0^{-1}f_1, \ldots, f_0^{-1}f_n)$ and each coefficient of $H$ lies in $\A(r)$. We get
\begin{eqnarray*}
  F & = & (F - G) + (G - G(f_0^{-1}f_1, \ldots, f_0^{-1}f_n)) - (H - H(f_0^{-1}f_1, \ldots, f_0^{-1}f_n)) + H \\
  & \in & \A \Ens{T_1, \ldots, T_n}(r) + (f_0T_1 - f_1, \ldots, f_0T_n - f_n),
\end{eqnarray*}
and hence $\n{F + (f_0T_1 - f_1, \ldots, f_0T_n - f_n)\hat{}}_{\A \ens{f_0^{-1}f}} \leq r$. Since $r$ is an arbitrary real number with $r > \n{\widetilde{\iota}(F)}_{\A \langle f_0^{-1}f \rangle}$, it implies $\n{F + (f_0T_1 - f_1, \ldots, f_0T_n - f_n)\hat{}}_{\A \ens{f_0^{-1}f}} \leq \n{\widetilde{\iota}(F)}_{\A \langle f_0^{-1}f \rangle}$. Thus $\iota \ens{f_0^{-1}f}$ is isometric.
\end{proof}

Let $\A$ be an affinoid ring. A {\it rational domain of $\t{Spa}(\A)$} is a subset of $\t{Spa}(\A)$ of the form $\t{Spa}(\A)(f_0^{-1}f) \coloneqq \set{x \in \t{Spa}(\A)}{\v{f_i(x)} \leq \v{f_0(x)}, {}^{\forall} i \in \N \cap [1,n]}$ for some $n \in \N$, $f_0 \in \A^{\triangleright}$, and $f = (f_1, \ldots, f_n) \in (\A^{\triangleright})^n$ with $\sum_{i = 1}^{n} \A^{\triangleright} f_i = \A^{\triangleright}$. We adopt obvious conventions such as $\t{Spa}(\A)(f)$. For a rational subset $U \subset \t{Spa}(\A)$, an affinoid ring $(\Opn_{\t{Spa}(\A)}(U), \Opn_{\t{Spa}(\A)}^{+}(U))$ over $\A$ is defined by a certain universality (\cite{Hub2} (1.2)) so that the structure morphism $\A \to (\Opn_{\t{Spa}(\A)}(U), \Opn_{\t{Spa}(\A)}^{+}(U))$ induces an identification $\t{Spa}(\Opn_{\t{Spa}(\A)}(U), \Opn_{\t{Spa}(\A)}^{+}(U)) \cong U \subset \t{Spa}(\A)$, and is called {\it the rational localisation of $\A$ on $U$}. For any presentation $U = \t{Spa}(\A)(f_0^{-1}f)$ by an $(n,f_0,f)$, the universality guarantees that $\Opn_{\t{Spa}(\A)}(U)$ is isomorphic to $\A \langle f_0^{-1}f \rangle$ over $\A$. Moreover, $\Opn_{\t{Spa}(\A)}$ and $\Opn_{\t{Spa}(\A)}^{+}$ give presheaves of topological rings on the topology of $\t{Spa}(\A)$ generated by rational domains.

\begin{prp}
\label{comparison}
Let $\A$ be an affinoid ring over $k$. Put $X \coloneqq \t{Spa}(\A)$. Let $U \subset X$ be a rational domain, and $\n{\cdot}_{\A}$ a complete norm on $\A^{\triangleright}$ as a $k$-algebra which gives its original topology. For any $n \in \N \backslash \ens{0}$, $f_0 \in \A$, and $f = (f_1,\ldots,f_n) \in \A^n$ with $\sum_{i = 1}^{n} \A f_i = \A$ and $U = X(f_0^{-1}f)$, there is a natural isomorphism $\A^{\triangleright} \ens{f_0^{-1}f} \to \Opn_{X}(U)$ as topological $k$-algebras.
\end{prp}

\begin{proof}
The assertion is an immediate consequence of Proposition \ref{presentation}.
\end{proof}

\begin{prp}
\label{presentation 2}
Let $\A$ be a Banach function algebra over $k$. For any $f \in \A$, $(T - f) \subset \A \ens{T}$ is a closed ideal, and in addition if $\A$ is uniform, then the canonical bounded homomorphism $\A \ens{T}/(T - f) \to \A \ens{f}$ is an isometric isomorphism.
\end{prp}

\begin{proof}
We have only to verify that $(T - f)$ is a closed ideal. Since $\A$ is Banach function algebra over $k$, the uniformisation preserves the underlying topological $k$-algebra. Therefore we may assume that $\A$ is uniform. Let $F \in \A \ens{T}$. For each $x \in \M(\A)$, denote by $\tilde{x} \in \M(\A \ens{T})$ the point given by setting $\v{G(\tilde{x})} = \sup_{i = 0}^{\infty} \v{G_i(\tilde{x})}$ for each $G = \sum_{i = 0}^{\infty} G_i \mathbb{T}^i \in \A \ens{T}$. We have $\v{(T - f)(\tilde{x})} = \max \ens{1, \v{f(x)}} \geq 1$, and hence $\v{F(\tilde{x})} \leq \v{(T - f)(\tilde{x})} \ \v{F(\tilde{x})} = \v{(T - f)F(\tilde{x})}$ for any $x \in \M(\A)$. The uniformity of $\A$ implies $\n{F}_{\A \ens{T}} \leq \n{(T - f)F}_{\A \ens{T}} \leq \n{T - f}_{\A \ens{T}} \ \n{F}_{\A \ens{T}}$ by \cite{Ber} 1.3.1 Theorem, because $\tilde{x}$ is maximal among the preimage of $x$. Thus the multiplication $(T - f) \times \colon \A \ens{T} \to \A \ens{T}$ is an admissible injective homomorphism, and its image $(T - f)$ is closed.
\end{proof}

\begin{prp}
\label{presentation 3}
Let $\A$ be a Banach function algebra over $k$. For any $f \in \A$, $(fT - 1) \subset \A \ens{T}$ is a closed ideal, and in addition if $\A$ is uniform, then the canonical bounded homomorphism $\A \ens{T}/(fT - 1) \to \A \ens{f^{-1}}$ is an isometric isomorphism.
\end{prp}

\begin{proof}
The assertion is verified in a similar way to Proposition \ref{presentation 2}.
\end{proof}

As a consequence, if $\A$ is a Banach function algebra over $k$, then the ideal $(T_1T_2 - 1) \subset \A \ens{T_1,T_2} \cong \A \ens{T_1} \ens{T_2}$ is closed. Therefore the bounded homomorphism $\A \ens{T_1,T_2}/(T_1T_2 - 1) \to \A \ens{T,T^{-1}}$ sending the image of $T_1$ to $T$ is an isometric isomorphism. In fact, this presentation does not need the assumption of the uniformity of $\A$ because the multiplication of $T_1T_2 - 1 \in \A \ens{T_1,T_2}$ is an isometry by the definition of the Gauss norm. It is easy to verify that $\A \ens{T,T^{-1}}$ is isometrically isomorphic to the Banach $k$-algebra of convergent Laurent series $F = \sum_{i \in \Z} F_i T^i$ with $\lim_{\v{i} \to \infty} \n{F_i}_{\A} = 0$. Therefore if $\A$ is a Banach function algebra over $k$ (resp.\ a uniform Banach $k$-algebra), then so is $\A \ens{T,T^{-1}}$. By the definition of $\A \ens{T,T^{-1}}$, the $\check{\m{C}}$ech complex
\begin{eqnarray*}
\begin{array}{ccccccccc}
  0 & \to & \A & \to & \A \Ens{T_1} \times \A \Ens{T_2} & \to & \A \Ens{T,T^{-1}} & \to & 0 \\
  && f & \mapsto & (f,f) &&&& \\
  &&&& (F(T_1),G(T_2)) & \mapsto & F(T) - G(T^{-1}) &&
\end{array}
\end{eqnarray*}
is an admissible exact sequence of Banach $k$-vector spaces.

\begin{prp}
\label{presentation 4}
Let $\A$ be a Banach function algebra over $k$. For any $f \in \A$, the kernel of the canonical homomorphism $\A \ens{T,T^{-1}} \to \A \ens{f,f^{-1}}$ sending $T$ to the image of $f$ coincides with $(T - f) \subset \A \ens{T,T^{-1}}$. In addition if $\A$ is uniform, the canonical bounded homomorphism $\A \ens{T,T^{-1}}/(T - f) \to \A \ens{f,f^{-1}}$ is an isometric isomorphism.
\end{prp}

\begin{proof}
The assertion immediately follows from Proposition \ref{presentation 2} applied to the composite of the canonical projection $\A \ens{T_1,T_2} \twoheadrightarrow \A \ens{T,T^{-1}}$ and the given homomorphism $\A \ens{T,T^{-1}} \twoheadrightarrow \A \ens{f,f^{-1}}$.
\end{proof}

\begin{crl}
\label{acyclicity}
Let $\A$ be a Banach function algebra over $k$. For any $f \in \A$, the $\check{\m{C}}$ech complex $0 \to \A \to \A \ens{f} \times \A \ens{f^{-1}} \to \A \ens{f,f^{-1}} \to 0$ is an admissible exact sequence.
\end{crl}

\begin{proof}
Consider the commutative diagram
\begin{eqnarray*}
\begin{CD}
    @.   0  @.   0                                 @.   0                 @. \\
  @.     @VVV    @VVV                                   @VVV \\
  0 @>>> 0  @>>> (T_1 - f) \times (fT_2 - 1)       @>>> (T - f)           @>>> 0 \\
  @.     @VVV    @VVV                                   @VVV \\
  0 @>>> \A @>>> \A \Ens{T_1} \times \A \Ens{T_2}  @>>> \A \Ens{T,T^{-1}} @>>> 0 \\
  @.     @VVV    @VVV                                   @VVV \\
  0 @>>> \A @>>> \A \Ens{f} \times \A \Ens{f^{-1}} @>>> \A \Ens{f,f^{-1}} @>>> 0 \\
  @.     @VVV    @VVV                                   @VVV \\
    @.   0  @.   0                                 @.   0                 @.
\end{CD}
\end{eqnarray*}
whose arrows are the obvious ones. By Proposition \ref{presentation 2}, Proposition \ref{presentation 3}, and Proposition \ref{presentation 4}, each column is exact. The central row is exact as is mentioned after Proposition \ref{presentation 3}. The exactness of the top row is verified in a purely algebraic elementary calculation. Therefore the exactness of the bottom row follows from the snake lemma. Banach's open mapping theorem (\cite{Bou} Theorem I.3.3/1) guarantees the admissibility of it.
\end{proof}

\begin{lmm}
\label{transitivity}
Let $\A$ be a Banach $k$-algebra. For any $n \in \N \backslash \ens{0}$ and $f = (f_1,\ldots,f_n) \in \A^n$, the bounded homomorphism $\A \ens{T_1,\ldots,T_n} \to \A \ens{f_1} \cdots \ens{f_n}$ sending $T_i$ to $f_i$ for each $i \in \N \cap [1,n]$ induces an isometric isomorphism $\A \ens{f} \to \A \ens{f_1} \cdots \ens{f_n}$.
\end{lmm}

\begin{proof}
To begin with, we verify that for any closed ideal $I \subset \A$, the canonical bounded homomorphism $\A \ens{T}/I \ens{T} \to \A/I \ens{T}$ is an isometric isomorphism, where $I \ens{T} \subset \A \ens{T}$ is the closed ideal generated by the image of $I$. Since the canonical projection $\A \twoheadrightarrow \A/I \ens{T}$ is submetric, so is the induced homomorphism $\A \ens{T} \to \A/I \ens{T}$. Its kernel is the closed ideal consisting of power series $F = \sum_{i = 0}^{\infty} F_i T^i \in \A \ens{T}$ with $F_i \in I$ for any $i \in \N$, and coincides with $I \ens{T}$. Therefore $\A \ens{T}/I \ens{T} \to \A/I \ens{T}$ is injective. Let $\overline{F} = \sum_{i = 0}^{\infty} \overline{F}_i T^i \in \A/I \ens{T}$. Take an arbitrary $\epsilon \in (0,1)$. Let $F_1, F_2, \ldots \in \A$ be lifts of $\overline{F}_1, \overline{F}_2, \ldots$ such that $\n{F_i}_{\A} \leq \n{\overline{F}_i}_{\A/I} + \epsilon^i$ for any $i \in \N$. Then $0 \leq \lim_{i \to \infty} \n{F_i}_{\A} \leq \lim_{i \to \infty} \n{\overline{F}_i}_{\A/I} + \epsilon^i = 0$, and hence $F \coloneqq \sum_{i = 0} F_i T^i \in \A[[T]]$ lies in $\A \ens{T}$. The image of $F$ in $\A/I \ens{T}$ is $\overline{F}$, and $\n{F}_{\A \ens{T}} \leq \n{\overline{F}}_{\A/I \ens{T}} + \epsilon$. Thus we obtain the surjectivity of $\A \ens{T}/I \ens{T} \to \A/I \ens{T}$, and the estimation of the norms above guarantees that it is an isometry. Applying the result above to $\A \ens{f} \to \A \ens{f_1} \cdots \ens{f_n}$ repetitively, we get
\begin{eqnarray*}
  & & \A \Ens{f_1, \ldots, f_n} = \A \Ens{T_1, \ldots, T_n}/(T_1 - f_1, \ldots, T_n - f_n)\hat{} \\
  & \cong & \left( \A \Ens{T_1} \cdots \Ens{T_n}/(T_1 - f_1)\hat{} \right)/(T_2 - f_2, \ldots, T_n - f_n)\hat{} \\
  & \cong & \left( \A \Ens{T_1} \cdots \Ens{T_{n-1}}/(T_1 - f_1)\hat{} \right) \Ens{T_n} /(T_2 - f_2, \ldots, T_n - f_n)\hat{} \\
  & \cong & \cdots \cong \left( \A \Ens{T_1}/(T_1 - f_1)\hat{} \right) \Ens{T_2} \cdots \Ens{T_n}/(T_2 - f_2, \ldots, T_n - f_n)\hat{} \\
  & = & \A \Ens{f_1} \Ens{T_2} \cdots \Ens{T_n}/(T_2 - f_2, \ldots, T_n - f_n)\hat{} \cong \cdots \cong \A \Ens{f_1} \cdots \Ens{f_n},
\end{eqnarray*}
where every isomorphism is an isometry.
\end{proof}

\begin{dfn}
\label{special}
Let $\A$ be an affinoid ring over $k$. A rational covering $\U$ of $X \coloneqq \t{Spa}(\A)$ is said to be {\it cyclic} (resp.\ {\it strongly cyclic}) if there are some $n \in \N \backslash \ens{0}$ and $f = (f_1, \ldots, f_n) \in \A^n$ (resp.\ $f = (f_1, \ldots, f_n) \in (\A^{\times})^n$) such that $\sum_{i = 1}^{n} \A^{\triangleright} f_i = \A^{\triangleright}$ and $\U \subset \set{X(f_i^{-1}f) \hookrightarrow X}{i \in \N \cap [1,n]}$, and is said to be {\it special} if there are an $n \in \N \backslash \ens{0}$ and $f = (f_1, \ldots, f_n) \in (\A^{\times})^n$ such that $\U \subset \set{X(f^{\sigma}) \hookrightarrow X}{\sigma \in \ens{1, -1}^n}$.
\end{dfn}

\begin{lmm}
\label{refinement}
Let $\A$ be an affinoid ring over $k$, and put $X \coloneqq \t{Spa}(\A)$
\begin{itemize}
\item[(i)] Every rational covering $\U$ of $X$ admits a cyclic covering of $X$ refining $\U$.
\item[(ii)] Every cyclic covering $\V$ of $X$ admits a special covering $\W$ of $X$ such that $\V |_W \coloneqq \set{V \cap W \hookrightarrow W}{(V \hookrightarrow X) \in \V}$ is strongly cyclic for any $(W \hookrightarrow X) \in \W$.
\item[(iii)] Every strongly cyclic covering $\W$ of $X$ admits a special covering of $X$ refining $\W$.
\end{itemize}
\end{lmm}

This is a counter part of \cite{BGR} Lemma 8.2.2/2, Lemma 8.2.2/3, and Lemma 8.2.2/4 for an adic space, and the following is just imitating the proofs of them.

\begin{proof}
Let $\U = \ens{U_1, \ldots, U_n \hookrightarrow X}$ be a rational covering. Take an $m_i \in \N \backslash \ens{0}$, an $f_{i,0} \in \A^{\triangleright}$, and $f_i = (f_{i,1},\ldots,f_{i,m_i}) \in (\A^{\triangleright})^n$ with $\sum_{j = 1}^{m_i} \A^{\triangleright} f_{i,j} = \A^{\triangleright}$ and $U_i = X(f_{i,0}^{-1}f_i)$ for each $i \in \N \cap [1,n]$. For each $J = (J_1, \ldots, J_n) \in \N^n \cap \prod_{i = 1}^{n} [0,m_i]$, put $f_J \coloneqq \prod_{i = 1}^{n} f_{i,J_i} \in \A^{\triangleright}$ and set $V_J \coloneqq \set{x \in X}{\v{f_{J'}(x)} \leq \v{f_J(x)}, {}^{\forall} J' \in \N^n \cap \prod_{i = 1}^{n} [0,m_i]}$. Then $\V \coloneqq \set{V_J}{J \in \N^n \cap \prod_{i = 1}^{n} [0,m_i]}$ is a cyclic covering refining $\U$.

\vspace{0.2in}
Let $\V$ be a cyclic covering of $X$. Replacing $\V$ by a rational refinement $\ens{V_1, \ldots, V_n \hookrightarrow X}$ obtained by adding suitable rational domains of $X$ to $\V$, we may assume that there is a $g =(g_1,\ldots,g_n) \in (\A^{\triangleright})^n$ such that $\sum_{j = 1}^{n} \A^{\triangleright} g_j = \A^{\triangleright}$ and $V_j = X(g_j^{-1}g)$ for each $j \in \N \cap [1,n]$. Let $j \in \N \cap [1,n]$. Denote by $\A_j$ the affinoid ring over $\A$ corresponding to $V_j$. The image of $g_j$ is invertible in $\A_j^{\triangleright} = \t{H}^0(V_j,\Opn_X)$. Set $R_j \coloneqq \n{g_j^{-1}}_{\A_j}^{-1} > 0$. Put $R \coloneqq \min_{j \in \N \cap [1,n]} R_j > 0$. Since the valuation of $k$ is non-trivial, there is a $c \in k^{\times}$ with $\v{c} < R$. For each $F \subset \N \cap [1,n]$, set $W_F \coloneqq \set{x \in X}{\v{(c^{-1} g_j)(x)} \leq 1, {}^{\forall} j \in F} \cap \set{x \in X}{\v{(c^{-1} g_j)(x)} \geq 1, {}^{\forall} j \in S \backslash F}$. We have $W_{\N \cap [1,n]} \cap X_1 = \emptyset$ by the choice of $c$, and the f-adic ring $\t{H}^0(W_{\N \cap [1,n]}, \Opn_X)$ is $0$ by Proposition \ref{height 1} and \cite{Ber} 1.2.1 Theorem. Therefore $W_{\N \cap [1,n]} = \emptyset$, and $\W \coloneqq \set{W_F}{F \subsetneq \N \cap [1,n]}$ is a rational covering. For each $(j,F) \in S \coloneqq (\N \cap [1,n]) \times (2^{\N \cap [1,n]} \backslash \ens{\N \cap [1,n]})$, we set $V_{j,F} \coloneqq V_j \cap W_F$. Put $\Sym \coloneqq \set{(j,F) \in S}{j \notin F}$. For any $(j,F) \in S \backslash \Sym$, we get $V_{j,F} = \emptyset$. It implies that $\V \W \coloneqq \set{V_{j,F}}{(j,F) \in \Sym}$ is a rational covering refining $\V$. For any $(j,F) \in \Sym$, the image of $f_j$ is invertible in $\t{H}^0(W_F, \Opn_X)$, and we have $V_{j,F} = \set{x \in W_F}{\v{(f_j^{-1}f_{j'})(x)} \leq 1, {}^{\forall} j' \in (\N \cap [1,n]) \backslash F}$ as a rational domain of $W_F$. In particular, for any $F \subsetneq \N \cap [1,n]$, $\V |_{W_F} \coloneqq \set{V_{j,F}}{(j,F) \in \Sym} \subset \V \W$ is a strongly cyclic covering.

\vspace{0.2in}
Let $\W$ is a strongly cyclic covering of $X$. Replacing $\W$ by a rational refinement $\ens{W_1, \ldots, W_n \\
%
%
\hookrightarrow X}$ obtained by adding suitable rational domains of $X$ to $\W$, we may assume that there is an $h = (h_1,\ldots,h_n) \in ((\A^{\triangleright})^{\times})^n$ such that $W_l = X(h_l^{-1}h)$ for each $l \in \N \cap [1,n]$. Denote by $\T \subset (\N \cap [1,n]) \times 2^{\N \cap [1,n]}$ the subset of elements $(l,\sigma)$ with $l \notin \sigma$. For each $(l,\sigma) \in \T$, put $W_{l,\sigma} \coloneqq \set{x \in W_l}{\v{h_l(x)} \leq 1, {}^{\forall} l \in \sigma} \cap \set{x \in W_l}{\v{h_l^{-1}(x)} \leq 1, {}^{\forall} l \in (\N \cap [1,n]) \backslash \sigma}$. Then for each $l \in \N \cap [1,n]$, $\set{W_{l,\sigma} \hookrightarrow W_l}{(l,\sigma) \in \T}$ is a special covering, and $\set{W_{l,\sigma} \hookrightarrow X}{(l,\sigma) \in \T}$ is a special covering refining $\W$.
\end{proof}

\section{Uniformity of Berkovich Spectra}
\label{Uniformity of Berkovich Spectra}

We construct a uniform Banach $k$-algebra whose certain rational localisation is not uniform in \S \ref{Negative Facts for Berkovich Spaces}. We also give several affirmative results for a rational localisation of a Berkovich spectrum in \S \ref{Affirmative Facts for Berkovich Spaces}.

\subsection{Negative Facts}
\label{Negative Facts for Berkovich Spaces}

We give an example of a rational localisation which does not preserve the uniformity. Take a monotonously increasing sequence $(a_1, \ldots) \in [1,\infty)$ with $a_1 = 1$, $\lim_{i \to \infty} a_i = \infty$, and $\lim_{i \to \infty} (i+1)^{-1}a_{i+1} = 0$. Put $a_0 \coloneqq 1$. Let $r \in ( 0, \infty )$. We denote by $k \set{r^{-a_i}U^iX^{a_i}}{i \in \N} \subset k \ens{r^{-1}X,U}$ the closure of the $k$-subalgebra generated by $\set{U^iX^{a_i}}{i \in \N}$. Since $k \ens{r^{-1}X,U}$ is uniform, so is $k \set{r^{-a_i}U^iX^{a_i}}{i \in \N}$.

\begin{thm}
\label{not uniform 1}
The Banach $k$-algebra $k \set{r^{-a_i}U^iX^{a_i}}{i \in \N} \ens{X}$ is not a Banach function algebra over $k$.
\end{thm}

\begin{proof}
Put $\A \coloneqq k \set{r^{-a_i}U^iX^{a_i}}{i \in \N}$ $\B \coloneqq \A \ens{X}$. To begin with, we show that the natural inclusion $\iota \colon \A \hookrightarrow k \ens{r^{-1}X,U} \hookrightarrow k \ens{U}[[X]]$ is extended to a unique continuous $k$-algebra homomorphism $\iota \ens{X} \colon \B \hookrightarrow k \ens{U}[[X]]$ with respect to the inverse limit topology of $k \ens{U}[[X]]$. The inclusion $\iota$ is continuous because it is the restriction of the natural inclusion $k \ens{r^{-1}X,U} \hookrightarrow k \ens{U}[[X]]$ given through the isometric isomorphism $k \ens{r^{-1}X,U} \cong k \ens{U} \ens{r^{-1}X}$. The $k$-algebra homomorphism $\A[T] \to k \ens{U}[[X]]$ extending $\iota$ and sending $T$ to $X$ uniquely extends to a continuous homomorphism $\widetilde{\iota} \colon \A \ens{T} \to k \ens{U}[[X]]$ because $X$ is a topologically nilpotent element and $k \ens{U}[[X]]$ is a Fr\'echet $k$-algebra. Its kernel contains the principal ideal $(T - X)$, and hence it induces a continuous $k$-algebra homomorphism $\iota \Ens{X} \colon \B \to k \ens{U}[[X]]$ because $\widetilde{\iota}$ is continuous and $k \ens{U}[[X]]$ is Hausdorff. Put $b_0 \coloneqq a_0 - 1 = 0$ and $b_i = a_i$ for each $i \in \N$. Let $f = \sum_{h = 0}^{\infty} \sum_{i = 0}^{\infty} \sum_{j = b_i}^{\infty} f_{h,i,j}U^iX^jT^h \in \ker(\widetilde{\iota})$. Since $\lim_{h+i+j \to \infty} \v{f_{h,i,j}} r^j = 0$, the infinite sum $\sum_{i = 0}^{\infty} \sum_{l = b_i}^{\infty} \sum_{j = b_i}^{l} f_{l-j,i,j}U^iX^jT^{l-j}$ also converges to $f$. Therefore by the continuity of $\widetilde{\iota}$, we have $\sum_{i = 0}^{\infty} \sum_{l = b_i}^{\infty} \sum_{j = b_i}^{l} f_{l-j,i,j} U^iX^l = \widetilde{\iota}(f) = 0$. Comparing the coefficients of both sides, we obtain $0 = \sum_{j = b_i}^{l}  f_{l-j,i,j}$ for any $i,l \in \N$ with $l \geq b_i$. In particular, $f_{0,i,b_i} = 0$ for any $i \in \N$. Set $g \coloneqq - \sum_{h = 0}^{\infty} \sum_{i = 0}^{\infty} \sum_{j = b_i}^{\infty} \sum_{m = 0}^{h} f_{h-m,i,j+m+1} U^iX^jT^h \in k[[U,X,T]]$. The equality $\lim_{h+i+j \to \infty} \v{f_{h,i,j}} r^j = 0$ implies $\lim_{h \to \infty} \sup_{i \in \N} \sup_{j \geq b_i} \v{- \sum_{m = 0}^{h} f_{h-m,i,j+m+1}} r^j = 0$ because $r > 1$, and hence $g$ lies in $\A \ens{T}$. We have
\begin{eqnarray*}
  & & (T - X)g \\
  & = & \sum_{i = 0}^{\infty} \sum_{j = b_i}^{\infty} f_{0,i,j+1} U^iX^{j+1} - \sum_{h = 1}^{\infty} \sum_{i = 0}^{\infty} \sum_{m = 0}^{h - 1} f_{h-1-m,i,b_i+m+1} T(U^iX^{b_i}T^{h-1}) \\
  & & + \sum_{h = 1}^{\infty} \sum_{i = 0}^{\infty} \sum_{j = b_i + 1}^{\infty} \left( - \sum_{m = 0}^{h - 1} f_{h-1-m,i,j+m+1} + \sum_{m = 0}^{h} f_{h-m,i,j+m} \right) U^iX^jT^h \\
  & = & \sum_{i = 0}^{\infty} \sum_{j = b_i + 1}^{\infty} f_{0,i,j} U^iX^j - \sum_{h = 1}^{\infty} \sum_{i = 0}^{\infty} \sum_{m = 0}^{h - 1} f_{h-1-m,i,b_i+m+1} U^iX^{b_i}T^h + \sum_{h = 1}^{\infty} \sum_{i = 0}^{\infty} \sum_{j = b_i + 1}^{\infty} f_{h,i,j} U^iX^jT^h \\
  & = & \sum_{i = 0}^{\infty} \left( 0 + \sum_{j = b_i + 1}^{\infty} f_{0,i,j} U^iX^j \right) + \sum_{h = 1}^{\infty} \sum_{i = 0}^{\infty} \left( 0 - \sum_{m = 0}^{h-1} f_{h-1-m,i,b_i+m+1} \right) U^iX^{b_i}T^h \\
  & & + \sum_{h = 1}^{\infty} \sum_{i = 0}^{\infty} \sum_{j = b_i + 1}^{\infty} f_{h,i,j} U^iX^jT^h \\
  & = & \sum_{i = 0}^{\infty} \left( f_{0,i,b_i} U^iX^{b_i} + \sum_{j = b_i + 1}^{\infty} f_{0,i,j} U^iX^j \right) \\
  & & + \sum_{h = 1}^{\infty} \sum_{i = 0}^{\infty} \left( \sum_{j = b_i}^{b_i + h} f_{b_i+h-j,i,j} - \sum_{m = 0}^{h-1} f_{h-1-m,i,b_i+m+1} \right) U^iX^{b_i}T^h + \sum_{h = 1}^{\infty} \sum_{i = 0}^{\infty} \sum_{j = b_i + 1}^{\infty} f_{h,i,j} U^iX^jT^h \\
  & = & \sum_{i = 0}^{\infty} \sum_{j = b_i}^{\infty} f_{0,i,j} U^iX^j + \sum_{h = 1}^{\infty} \sum_{i = 0}^{\infty} f_{h,i,b_i} U^iX^{b_i}T^h + \sum_{h = 1}^{\infty} \sum_{i = 0}^{\infty} \sum_{j = b_i + 1}^{\infty} f_{h,i,j} U^iX^jT^h \\
  & = & \sum_{h = 0}^{\infty} \sum_{i = 0}^{\infty} \sum_{j = b_i}^{\infty} f_{h,i,j} U^iX^jT^h = f.
\end{eqnarray*}
Therefore $f \in (T - X)$. We obtain $\ker(\widetilde{\iota}) = (T - X)$, and hence $\iota \ens{X}$ is injective.

\vspace{0.2in}
Since the image of $k[U^iX^{a_i} \mid i \in \N]$ is dense in $\B$, the image of $\iota \ens{X}$ is contained in the closed $k$-subalgebra $\B \coloneqq \set{F = \sum_{i = 0}^{\infty} \sum_{j = 0}^{\infty} F_{i,j} U^iX^j \in k \ens{U}[[X]]}{F_{i,j} = 0, {}^{\forall} j < b_i}$. Denote by $\varphi \colon \A \to \B$ the canonical bounded homomorphism. We verify $\n{\varphi(U^iX^j)}_{\B} = r^{b_i}$ for any $i,j \in \N$ with $j \geq b_i$. Let $i_0,j_0 \in \N$ with $j_0 \geq b_{i_0}$. We have $\n{\varphi(U^{i_0}X^{j_0})}_{\B} \leq \n{\varphi(X^{j_0-b_{i_0}})}_{\B} \n{\varphi(U^{i_0}X^{b_{i_0}})}_{\B} \leq r^{b_{i_0}}$. Assume $\n{\varphi(U^{i_0}X^{j_0})}_{\B} < r^{b_{i_0}}$. Then there are an $F \in \A \ens{T}$ and a $G \in \A \ens{T}(r^{b_{i_0}}-)$ such that $U^{i_0}X^{j_0} = (T - X)F + G$. Set $F \coloneqq \sum_{h = 0}^{\infty} \sum_{i = 0}^{\infty} \sum_{j = b_i}^{\infty} F_{h,i,j} U^iX^jT^h$ and $G \coloneqq \sum_{h = 0}^{\infty} \sum_{i = 0}^{\infty} \sum_{j = b_i}^{\infty} G_{h,i,j} U^iX^jT^h$. We have
\begin{eqnarray}
  U^{i_0}X^{j_0} & = & - X \sum_{i = 0}^{\infty} \sum_{j = b_i}^{\infty} F_{0,i,j} U^iX^j + \sum_{i = 0}^{\infty} \sum_{j = b_i}^{\infty} G_{0,i,j} U^iX^j \nonumber \\
  & = & \sum_{i = 0}^{\infty} \left( G_{0,i,b_i} U^iX^{b_i} + \sum_{j = b_i + 1}^{\infty} (-F_{0,i,j-1} + G_{0,i,j}) U^iX^j \right)
\end{eqnarray}
and
\begin{eqnarray}
  0 & = & \sum_{i = 0}^{\infty} F_{h-1,i,j} U^iX^{b_i} + \sum_{i = 0}^{\infty} \sum_{j = b_i + 1}^{\infty} (F_{h-1,i,j} - F_{h,i,j-1}) U^iX^j + \sum_{i = 0}^{\infty} \sum_{j = b_i}^{\infty} G_{h,i,j} U^iX^j \nonumber \\
  & = & \sum_{i = 0}^{\infty} (F_{h-1,i,b_i} + G_{h,i,b_i}) U^iX^{b_i} + \sum_{i = 0}^{\infty} \sum_{j = b_i + 1}^{\infty} (F_{h-1,i,j} - F_{h,i,j-1} + G_{h,i,j}) U^iX^j
\end{eqnarray}
for any $h \in \N \backslash \ens{0}$. Assume $j_0 = b_{i_0}$. Comparing the coefficients of the equality $(1)$ through the inclusion $\iota \ens{X}$, we obtain $G_{0,i_0,b_{i_0}} = 1$. It contradicts the inequality $\n{G_{0,i_0,b_{i_0}} U^{i_0}X^{b_{i_0}}}_{\A} \leq \n{G}_{\A \ens{T}} < \n{U^{i_0} X^{b_{i_0}}}_{\A} r^{b_{i_0}}$. Therefore we get $j_0 > b_{i_0}$. Comparing the coefficients of both sides of the equalities $(1)$ and $(2)$, we obtain $G_{0,i,b_i} = 0$ for any $i \in \N$, $F_{0,i_0,j_0-1} + 1 = G_{0,i_0,j_0}$, $F_{0,i,j-1} = G_{0,i,j}$ for any $i,j \in \N$ with $(i,j) \neq (i_0,j_0)$ and $j \geq b_i + 1$, $F_{h-1,i,b_i} + G_{h,i,b_i} = 0$ for any $h,i,j \in \N$ with $h \geq 1$, and $F_{h-1,i,j} - F_{h,i,j-1} + G_{h,i,j} = 0$ for any $h,i,j \in \N$ with $h \geq 1$ and $j \geq b_i + 1$. We have $\v{G_{h,i_0,j_0}} = \n{G_{h,i_0,j_0} U^{i_0}X^{j_0}T^h}_{\A \ens{T}} r^{-j_0} \leq \n{G}_{\A \ens{T}} r^{-j_0} < r^{-(j_0-b_0)} < 1$ for any $h \in \N$, $\v{F_{0,i_0,j_0-1}} = \v{G_{0,i_0,j_0} - 1} = 1$, and $\v{F_{h,i_0,j_0-1}} = \v{F_{h-1,i_0,j_0} + G_{h,i_0,j_0}} = 1$ for any $h \in \N \backslash \ens{0}$ by induction on $h$. It contradicts the fact $\lim_{h \to \infty} \sup_{i = 0}^{\infty} \sup_{j = b_i}^{\infty} \v{F_{h,i,j}} r^j = 0$. We obtain $\n{\varphi(U^{i_0}X^{j_0})}_{\B} = r^{b_{i_0}}$, and hence $\lim_{n \to \infty} \n{\varphi(UX)^n}_{\B}^{1/n} = \lim_{n \to \infty} r^{b_n/n} = r^{\lim_{n \to \infty} b_n/n} = 1$. Thus the set $\set{\n{f}_{\B}^{-1} \rho_{\B}(f)}{f \in \B \backslash \ens{0}}$ is unbounded in $[0,\infty)$, and $\n{\cdot}_{\B}$ is not equivalent to $\rho_{\B}$ as seminorms. We conclude that $\B$ is not a Banach function algebra over $k$.
\end{proof}

\begin{crl}
\label{not sheafy 1}
The uniform Banach $k$-algebra $k \set{r^{-a_i}U^iX^{a_i}}{i \in \N}$ is not sheafy.
\end{crl}

\begin{proof}
By the proof of Theorem \ref{not uniform 1}, we have $\n{\varphi(UX)^n}_{k \set{r^{-a_i}U^iX^{a_i}}{i \in \N} \ens{X}} = r^{b_n}$ for any $n \in \N$, while the norm of the image of $(UX)^n$ in $k \set{r^{-a_i}U^iX^{a_i}}{i \in \N} \ens{X, UX}$ is smaller than or equal to $1$. Thus the set
\begin{eqnarray*}
  \Set{\frac{\nn{f}_{k \Set{r^{-a_i}U^iX^{a_i}}{i \in \N} \Ens{X, UX}}}{\nn{f}_{k \set{r^{-a_i}U^iX^{a_i}}{i \in \N} \ens{X}}}}{f \in k \Set{r^{-a_i}U^iX^{a_i}}{i \in \N}, \nn{f}_{k \Set{r^{-a_i}U^iX^{a_i}}{i \in \N} \Ens{X}} > 0}
\end{eqnarray*}
is unbounded in $[0,\infty)$, and $k \set{r^{-a_i}U^iX^{a_i}}{i \in \N} \ens{X}$ and $k \set{r^{-a_i}U^iX^{a_i}}{i \in \N} \ens{X, UX}$ are not isomorphic to each other over $k \set{r^{-a_i}U^iX^{a_i}}{i \in \N}$, while they correspond to the same rational domain of $\M(k \set{r^{-a_i}U^iX^{a_i}}{i \in \N})$.
\end{proof}

\subsection{Affirmative Facts}
\label{Affirmative Facts for Berkovich Spaces}

\begin{dfn}
A Banach $k$-algebra $\A$ is said to be a {\it strongly Banach function algebra over $k$} if for any $n \in \N \backslash \ens{0}$ and $f \in \A^n$, $\A \ens{f}$ is a Banach function algebra over $k$.
\end{dfn}

For each $(m,l) \in \N \times \N$, we set $P_m^l \coloneqq (\ens{1, -1}^m)^l$. We denote by $Q_m^l \subset P_m^l$ the subset of elements $(\sigma_1, \ldots, \sigma_l)$ with $\sigma_1(m) = \cdots = \sigma_l(m)$. We put
\begin{eqnarray*}
  \A \Ens{f^{\Sigma}} & \coloneqq & \A \Ens{f_1^{\sigma_1(1)}, \ldots, f_n^{\sigma_1(n)}, f_1^{\sigma_2(1)}, \ldots, f_n^{\sigma_2(n)}, \ldots, f_1^{\sigma_l(1)}, \ldots, f_n^{\sigma_l(n)}} \\
  \langle \Sigma, 1 \rangle & \coloneqq & ((\sigma_1, 1), \ldots, (\sigma_l, 1)) \in Q_{n+1}^l \\
  \langle \Sigma, -1 \rangle & \coloneqq & ((\sigma_1, -1), \ldots, (\sigma_l, -1)) \in Q_{n+1}^l \\
  \langle \Sigma, \pm 1 \rangle & \coloneqq & ((\Sigma,1),(\Sigma,-1)) \in P_{n+1}^{2l}
\end{eqnarray*}
for each Banach $k$-algebra $\A$, $n \in \N \backslash \ens{0}$, $f = (f_1, \ldots, f_n) \in (\A^{\times})^n$, $l \in \N \cap [1,n]$, and $\Sigma = (\sigma_1, \ldots, \sigma_l) \in P_n^l$

\begin{thm}
\label{acyclicity 2}
Let $\A$ be a strongly Banach function algebra over $k$. For any $n \in \N \backslash \ens{0}$ and $f = (f_1, \ldots, f_n) \in (\A^{\times})^n$, the $\check{\m{C}}$ech complex
\begin{eqnarray*}
  0 \to \A \to \prod_{\sigma \in P_n^1} \A \Ens{f^{\Sigma}} \to \prod_{\Sigma \in P_n^2} A \Ens{f^{\Sigma}} \to \cdots \to \prod_{\Sigma \in P_n^{2^n}} A \Ens{f^{\Sigma}} \to 0
\end{eqnarray*}
is an admissible exact sequence.
\end{thm}

\begin{proof}
In the case $n = 1$, the assertion directly follows from Corollary \ref{acyclicity}. Let $m \in \N \backslash \ens{0}$. Suppose that and the assertion holds for any $n \in \N \backslash \ens{0}$ with $n \leq m$. We prove the assertion in the case $n = m+1$. Put $g \coloneqq (f_1, \ldots, f_m) \in (\A^{\times})^m$ and $h \coloneqq f_{m+1}$. Consider the following commutative diagram $D$:
\begin{eqnarray*}
\begin{CD}
    @. 0 @. 0 @. 0 @. \\
  @. @VVV @VVV @VVV @. \\
  0 @>>> \A @>>> \displaystyle\prod_{s \in \Ens{1, -1}} \A \Ens{h^s} @>>> \A \Ens{h,h^{-1}} @>>> 0 \\
  @. @VVV @VVV @VVV @. \\
  0 @>>> \displaystyle\prod_{\Sigma \in P_m^1} \A \Ens{g^{\Sigma}} @>>> \displaystyle\prod_{\Sigma \in Q_{m+1}^1} \A \Ens{f^{\Sigma}} @>>> \displaystyle\prod_{\Sigma \in P_m^1} \A \Ens{f^{\langle \Sigma, \pm 1 \rangle}} @>>> 0 \\
  @. @VVV @VVV @VVV @. \\
    @. \vdots @. \vdots @. \vdots @. \\
  @. @VVV @VVV @VVV @. \\
  0 @>>>  \displaystyle\prod_{\Sigma \in P_m^{2^m}} A \Ens{g^{\Sigma}} @>>> \displaystyle\prod_{\Sigma \in Q_{m+1}^{2^m}} \A \Ens{f^{\Sigma}} @>>> \displaystyle\prod_{\Sigma \in P_m^{2^m}} \A \Ens{f^{\langle \Sigma, \pm 1 \rangle}} @>>> 0 \\
  @. @VVV @VVV @VVV @. \\
    @. 0 @. 0 @. 0 @. .
\end{CD}
\end{eqnarray*}
Since $\A$ is a strongly Banach function algebra over $k$, each Banach algebra appearing in components of the diagram is a Banach function algebra. By Lemma \ref{transitivity}, we have $\A \ens{g^{\Sigma}} \ens{h^s}$ is isometrically isomorphic to $\A \ens{f^{\langle \Sigma, s \rangle}}$ for any $l \in \N \backslash \ens{0}$ and $(\Sigma,s) \in P_m^l \times \ens{1, -1}$. Therefore the central column (resp.\ the right column) is obtained by applying the Weierstrass localisations $\ens{h}$ and $\ens{h^{-1}}$ (resp.\ $\ens{h, h^{-1}}$) to each Banach $k$-algebra appearing in components of the left column. It implies that each row is an admissible exact sequence by Corollary \ref{acyclicity}. By the commutativity of Weierstrass localisations, the central column (resp.\ the right column) is the $\check{\m{C}}$ech complex with respect to the Weierstrass covering of $\M_k(\A \ens{h} \times \A \ens{h^{-1}})$ (resp.\ $\M_k(\A \ens{h, h^{-1}})$) associated to $f_1, \ldots, f_m, f_1^{-1}, \ldots, f_m^{-1}$. Therefore each column is an admissible exact sequence by the induction hypothesis. The total complex $\t{T}(D)$ of the diagram is naturally isometrically isomorphic to the $\check{\m{C}}$ech complex $C$ with respect to the Weierstrass covering of $\M_k(\A)$ associated to $f_1, \ldots, f_{m+1}, f_1^{-1}, \ldots, f_{m+1}^{-1}$. Thus the exactness of each row and each column of $D$ guarantees the exactness of $C$, and it is admissible by Banach's open mapping theorem (\cite{Bou} Theorem I.3.3/1).
\end{proof}

\begin{dfn}
Let $\A$ be a Banach $k$-algebra. An $a \in \A$ is said to be {\it quasi-nilpotent} if $\rho_{\A}(a) = 0$, and $\A$ is said to be {\it spectrally reduced} if no non-zero element of $\A$ is quasi-nilpotent.
\end{dfn}

A Banach $k$-algebra is spectrally reduced if and only if the Gel'fand transform (\cite{Ber} 1.2.2 Remarks (i)) of $\A$ on the Berkovich spectrum is injective.

\begin{dfn}
A Banach $k$-algebra is said to be a {\it Banach field over $k$} if its underlying $k$-algebra is a field.
\end{dfn}

Banach fields appear when we consider the Banach $k$-algebra of sections on a closed subset of the Berkovich spectrum associated to a Banach $k$-algebra topologically of finite type.

\begin{thm}
Let $\A$ be a uniform Banach field over $k$. For any $a \in \A$, $\A \ens{a}$ is spectrally reduced. Moreover, if the image of the map $\A^{\times} \to [0,\infty), b \mapsto \n{b}_{\A} \n{b^{-1}}_{\A}$ is bounded, then $\A \ens{a}$ is a Banach field over $k$ and the image of the map $\A \ens{a}^{\times} \to [0,\infty), b \mapsto \n{b}_{\A \ens{a}} \n{b^{-1}}_{\A \ens{a}}$ is bounded.
\end{thm}

\begin{proof}
Since $\A$ is uniform, $\A$ is spectrally reduced. If $\n{a}_{\A} \leq 1$, then $\A \ens{a} \cong \A$. Therefore we may assume $\n{a}_{\A} > 1$. In particular, $a \neq 0$ and hence $a \in \A^{\times}$. If $\A \ens{a} = 0$, then the assertion obviously holds. Suppose that $\A \ens{a} \neq 0$. Then Lemma \ref{Nullstellensatz} guarantees $\n{a^{-1}}_{\A} \geq 1$, because the uniformity of $\A$ and \cite{Ber} 1.3.1 Theorem imply $\n{a^{-1}}_{\A} = \rho_{\A}(a^{-1}) = \sup_{x \in \M(\A)} \v{a(x)}$. Let $n \in \A \ens{a}$ be quasi-nilpotent. We verify $n = 0$. Take a lift $f = \sum_{i = 0}^{\infty} f_i T^i \in \A \ens{T}$ of $n$ with respect to the canonical projection $\A \ens{T} \twoheadrightarrow \A \ens{a}$, and assume $f \neq 0$. Since $\A \ens{a} \neq 0$, $\M(\A \ens{a}) \neq \emptyset$ by \cite{Ber} 1.2.1 Theorem. Take an $x \in \M(\A \ens{a}) \cong \M(\A) \ens{a}$. Since $n$ is a quasi-nilpotent, $n(x) = 0 \in \kappa (x)$, and hence $f(x) = \sum_{i = 0} f_i(x) T^i \in \kappa (x) \ens{T} (T - a(x))$. The canonical homomorphism $\A \to \kappa (x)$ is injective with dense image because $\A$ is a field. Therefore $f(x) \neq 0$. Take a unique $g_x = \sum_{i = 0}^{\infty} g_{x,i} T^i \in \kappa (x) \ens{T}$ with $(T - a(x))g_x = f(x)$. We prove that $g_x$ lies in the image of $\A \ens{T}$. By the inequality $(T - a(x))g_x = f(x)$, we have $-a(x) g_{x,0} = f_0(x)$ and $g_{x,i} - a(x) g_{x,i+1} = f_{i+1}(x)$ for any $i \in \N$. Since $\v{a(x)} = \v{a^{-1}(x)}^{-1} \geq \n{a^{-1}}_{\A}^{-1} > 0$, $a(x) \neq 0 \in \kappa (x)$ and hence these equalities guarantee that $g_{x,i}$ lies in the image of $\A$ for any $i \in \N$. Let $g_i \in \A$ denote a unique element with $g_i(x) = g_{x,i}$ for each $i \in \N$. Set $g \coloneqq \sum_{i = 0}^{\infty} g_i T^i \in \A[[T]]$. Since $(T - a(x))g(x) = (T -a(x))g_x = f(x) \in \kappa (x) \ens{T} \subset \kappa (x)[[T]]$, we have $(T - a)g = f \in \A[[T]]$ by the injectivity of $\A \to \kappa (x)$. It implies that $f_{i + 1} = g_i - a g_{i + 1}$ and $g_i = - \sum_{j = 0}^{i} a^{-j-1} f_{i - j}$ for any $i \in \N$. We obtain that $g$coincides with $- \sum_{i = 0}^{\infty} (\sum_{j = 0}^{i} a^{-j-1}f_{i-j}) T^i$, and hence is independent of $x \in \M(\A \ens{a})$. In particular, $g$ satisfies $(T - a(y))g(y) = f(y) \in \kappa (y) \ens{T}$ for any $y \in \M(\A) \ens{a}$. Moreover, we have $T - a = -a (1 - a^{-1}T) \in \A^{\times} (1 + \A[[T]]T) \subset \A[[T]]^{\times}$, and hence $g = (T - a)^{-1}f \in \A[[T]]$.

\vspace{0.2in}
Let $\epsilon > 0$. We show that there is an $N \in \N$ such that $\n{g_i}_{\A \ens{a}} < \epsilon$ for any $i \geq N$. Since $f \in \A \ens{T}$, there is an $N \in \N$ such that $\n{f_i}_{\A} < \epsilon$ for any $i \geq N$. Assume that there is an $i \geq N$ such that $\n{g_i}_{\A \ens{a}} \geq \epsilon$. By the uniformity of $\A$, there is a $y \in \M(\A) \ens{a}$ such that $\v{g_i(y)} \geq \epsilon$. By the inequalities $\v{a(y)} \leq 1$ and $\v{f_{i + 1}(y)} \leq \n{f_{i + 1}}_{\A} < \epsilon \leq \v{g_i(y)}$, we have $\v{g_{i + 1}(y)} = \v{a(y)}^{-1} \v{(a g_{i + 1})(y)} = \v{a(y)}^{-1} \v{(g_i - f_{i + 1})(y)} \geq \epsilon$. Therefore $\overline{\lim}_{n \to \infty} \v{g_i(y)} \geq \epsilon$. It contradicts the fact that $g(y) = g_y \in \kappa (y) \ens{T}$. Thus $\n{g_i}_{\A \ens{a}} < \epsilon$ for any $i \geq N$. It implies that $\lim_{i \to \infty} \n{g_i}_{\A \ens{a}} = 0$ and the image of $g$ in $\A \ens{a} [[T]]$ lies in $\A \ens{a} \ens{T}$. In particular, the image of $f = (T-a)g$ in $\A \ens{a} \ens{T}$ lies in $(T-a) \A \ens{a} \ens{T}$. Since the canonical projection $\A \ens{T} \twoheadrightarrow \A \ens{a}$ coincides with the composite of the rational localisation $\A \ens{T} \to \A \ens{a} \ens{T}$ of coefficients and the evaluation $\A \ens{a} \ens{T} \twoheadrightarrow \A \ens{a} \colon T \mapsto a$, the image $n$ of $f$ in $\A \ens{a}$ is $0$. We conclude that $\A \ens{a}$ is spectrally reduced.

\vspace{0.2in}
Suppose that the image of the map $\A^{\times} \to [0,\infty), b \mapsto \n{b}_{\A} \n{b^{-1}}_{\A}$ is bounded. Set $C \coloneqq \sup_{b \in \A^{\times}} \n{b}_{\A} \n{b^{-1}}_{\A}$. For any $b \in \A \ens{a}^{\times}$ lying in the image of $\A$, we have $\n{b}_{\A \ens{a}} \n{b^{-1}_{\A \ens{a}}} \leq C$ because the canonical homomorphism $\A \to \A \ens{a}$ is submetric by the definition of $\n{\cdot}_{\A \ens{a}}$. Let $b \in \A \ens{a} \backslash \ens{0}$. Since the image of $\A$ is dense in $\A \ens{a}$, there is a sequence $(b_i)_{i \in \N} \in \A \ens{a}^{\N}$ such that each $b_i$ lies in the image of $\A$ and $\lim_{i \to \infty} b_i = b$. Since $b \neq 0$ and $\A \ens{a} \backslash \ens{0} \subset \A \ens{a}$ is open, we may assume $b_i \neq 0$ for any $i \in \N$. Since the underlying $k$-algebra of $\A$ is a field, $b_i \in \A \ens{a}^{\times}$ for any $i \in \N$. Let $\epsilon > 0$. Since $\lim_{i \to \infty} b_i = b$, there is an $N \in \N$ such that $\n{b - b_i}_{\A \ens{a}} < \min \ens{\n{b}_{\A \ens{a}},\epsilon}$ for any $i \geq N$. Therefore we have $\n{b_i^{-1} - b_j^{-1}}_{\A \ens{a}} \leq \n{b_j - b_i}_{\A \ens{a}} \n{b_i^{-1}}_{\A \ens{a}} \n{b_j^{-1}}_{\A \ens{a}} \leq \n{b_j - b_i}_{\A \ens{a}}(C \n{b_i}_{\A \ens{a}}^{-1})(C \n{b_j}_{\A \ens{a}}^{-1}) < C^2 \n{b}_{\A \ens{a}}^{-2} \epsilon$ for any $i,j \geq N$. It implies that $(b_i^{-1})_{i \in \N}$ converges, and the limit is the inverse of $b$ by the continuity of the multiplication. In particular, $b \in \A \ens{a}^{\times}$. Therefore $\A \ens{a}$ is a Banach filed over $k$. The inverse $(\cdot)^{-1} \colon \A \ens{a}^{\times} \to \A \ens{a}^{\times}$ is continuous by [BGR] Proposition 1.2.4/4, and hence the map $\A \ens{a}^{\times} \to [0,\infty), b \mapsto \n{b}_{\A \ens{a}} \n{b^{-1}}_{\A \ens{a}}$ is continuous. Since the image of $\A$ is dense in $\A \ens{a}$, we obtain $\sup_{b \in \A \ens{a}^{\times}} \n{b}_{\A \ens{a}} \n{b^{-1}}_{\A \ens{a}} \leq C$.
\end{proof}

For a topological space $X$, we denote by $\t{C}_{\t{bd}}(X,k)$ the uniform Banach $k$-algebra of bounded continuous $k$-valued functions on $X$ endowed with the supremum norm $\n{\cdot}_{\t{C}_{\t{bd}}(X,k)}$ on $X$. The evaluation $X \times \t{C}_{\t{bd}}(X,k) \to k \colon (x,f) \mapsto f(x)$ induces a continuous map $\iota_X \colon X \to \M(\t{C}_{\t{bd}}(X,k))$. If $X$ is compact, then then $\t{C}_{\t{bd}}(X,k)$ coincides with the $k$-algebra of continuous $k$-valued functions on $X$, and $\iota_X$ is surjective by Gel'fand--Naimark theorem (\cite{Ber} 9.2.5 Theorem (i)). If $X$ is zero-dimensional and Hausdorff, then $\iota_X$ is injective by the argument after \cite{Mih} Proposition 1.7. For a clopen subset $U \subset X$, the characteristic function $1_U \colon X \to k$ of $U$ gives a rational domain $\M(\t{C}_{\t{bd}}(\N,k)) \ens{1_U}$. The preimage of $\M(\t{C}_{\t{bd}}(\N,k)) \ens{1_U}$ by $\iota_X$ coincides with $U \subset X$ by the definition of $\iota_X$. We remark that a rational domain of $\M(\t{C}_{\t{bd}}(X,k))$ seems not to correspond to a clopen subset of $X$ in general even if $X$ is zero-dimensional and Hausdorff. There can be monstrously many points in $\M(\t{C}_{\t{bd}}(X,k))$ which are not $k$-rational when $X$ is not compact and $k$ is not a local field, i.e.\ a complete discrete valuation field with finite residue field. For more detail, see \cite{Mih} Theorem 4.12.

\begin{dfn}
A Banach $k$-algebra $\A$ is said to be {\it spectral} if $\M(\A)(k) \neq \emptyset$ and the equality $\n{a}_{\A} = \sup_{x \in \M(\A)(k)} \v{a(x)}$ holds for any $a \in \A$.
\end{dfn}

Every spectral Banach $k$-algebra is uniform, and is canonically isometrically isomorphic to a closed $k$-subalgebra of the Banach $k$-algebra $\t{C}_{\t{bd}}(\M(\A)(k),k)$ through the Gel'fand transform on $\M(\A)(k)$.

\begin{thm}
If $k$ is a local field, then every spectral $k$-algebra is sheafy.
\end{thm}

\begin{proof}
Let $\A$ be a spectral $k$-algebra. Since $\A$ is spectral, the image of the Gel'fand transform $\A \hookrightarrow \t{C}_{\t{bd}}(\M(\A)(k),k)$ separates points, i.e.\ for any distinct points $x,y \in \M(\A)(k)$, there is an $f \in \A$ such that $f(x) \neq f(y)$ as elements of $k$. By \cite{Mih} Theorem 4.19, there is a one-to-one correspondence between closed $k$-subalgebras of $\t{C}_{\t{bd}}(\M(\A)(k),k)$ separating points and totally disconnected Hausdorff quotients of $\M(\t{C}_{\t{bd}}(\M(\A)(k),k))$ which is faithful under $\M(\A)(k)$ in the sense of \cite{Mih} Definition 4.18. In particular, $\A$ is canonically isometrically isomorphic to $\t{C}_{\t{bd}}(X,k)$ for a totally disconnected compact Hausdorff space $X$. Therefore we may assume $\A = \t{C}_{\t{bd}}(X,k)$ for a compact space $X$ without loss of generality. Each $x \in X$ corresponds to a point of $\M(\t{C}_{\t{bd}}(X,k))$ by the evaluation at $x$. This correspondence gives a continuous map $\iota_X \colon X \to \M(\t{C}_{\t{bd}}(X,k))$. Since $X$ is a totally disconnected compact Hausdorff space, the natural map $X \to \M(\t{C}_{\t{bd}}(X,k))$ is a homeomorphism by Gel'fand--Naimark theorem (\cite{Ber} Theorem 9.2.5 (i)).

\vspace{0.2in}
Let $U \subset \M(\t{C}_{\t{bd}}(X,k))$ be a rational domain. Take generators $f_0, f_1, \ldots, f_n$ of the principal ideal $(1)$ with $U = \set{x \in \M(\t{C}_{\t{bd}}(X,k))}{\v{f_i(x)} \leq \v{f_0(x)}, {}^{\forall} i \in \N \cap [1,n]}$. Since $k$ is zero-dimensional, the subset $U \cap X \coloneqq \set{x \in X}{\v{f_i(x)} \leq \v{f_0(x)}, {}^{\forall} i \in \N \cap [1,n]}$ is clopen. Therefore $U \subset \M(\t{C}_{\t{bd}}(X,k))$ is clopen because it is the image of $U \cap X$ by $\iota_X$. Let $e \in \t{C}_{\t{bd}}(X,k)$ be the characteristic function of $U \cap X$. Since $1 - e$ vanishes at every point of $U$, $1 - e$ is contained in the kernel of the canonical bounded homomorphism $\t{C}_{\t{bd}}(X,k) \to \t{C}_{\t{bd}}(X,k) \ens{f_0^{-1}f_1,\ldots,f_0^{-1}f_n}$. Let $f \in \t{C}_{\t{bd}}(X,k)$. Since $\iota_X|_{U \cap X} \colon U \cap X \to U$ is surjective, we have $\n{f_U}_{\t{C}_{\t{bd}}(X,k) \ens{f_0^{-1}f_1,\ldots,f_0^{-1}f_n}} \geq \n{f|_{U \cap X}}_{\t{C}_{\t{bd}}(U \cap X,k)}$. On the other hand, we have $\n{f_U}_{\t{C}_{\t{bd}}(X,k) \ens{f_0^{-1}f_1,\ldots,f_0^{-1}f_n}} \leq \n{f|_{U \cap X}}_{\t{C}_{\t{bd}}(U \cap X,k)}$ by $\n{(1-e)f}_{\M(\t{C}_{\t{bd}}(X,k))} = \n{f|_{U \cap X}}_{\M(\t{C}_{\t{bd}}(U \cap X,k))}$. Thus $\n{f_U}_{\t{C}_{\t{bd}}(X,k) \ens{f_0^{-1}f_1,\ldots,f_0^{-1}f_n}} = \n{f|_{U \cap X}}_{\t{C}_{\t{bd}}(U \cap X,k)}$. Since $U \cap X \subset X$ is clopen, the restriction $(\cdot)|_{U \cap X} \colon \t{C}_{\t{bd}}(X,k) \to \t{C}_{\t{bd}}(U \cap X,k)$ is surjective. Therefore by the argument above, $\t{C}_{\t{bd}}(X,k) \ens{f_0^{-1}f_1,\ldots,f_0^{-1}f_n}$ and $\t{C}_{\t{bd}}(U \cap X,k)$ are isometrically isomorphic to each other. In particular, the isomorphism class of $\t{C}_{\t{bd}}(X,k) \ens{f_0^{-1}f_1,\ldots,f_0^{-1}f_n}$ is independent of the presentation $(n,(f_0,f_1,\ldots,f_n))$ of $U$, and the structure presheaf $\Opn_{\t{C}_{\t{bd}}(X,k)}$ on the Grothendieck topology generated by rational domains on $\M(\t{C}_{\t{bd}}(X,k))$ is well-defined. Moreover, $\Opn_{\t{C}_{\t{bd}}(X,k)}$ is canonically isomorphic to the sheaf of bounded continuous functions. We conclude that $\A = \t{C}_{\t{bd}}(X,k)$ is sheafy.
\end{proof}

\section{Uniformity of Adic Spectra}
\label{Uniformity of Adic Spectra}

We verify that the example in \S \ref{Negative Facts for Berkovich Spaces} yields examples of uniform affinoid rings whose rational localisations are not uniform in \S \ref{Negative Facts for Adic Spaces}. We verify that one of them is an example of a non-sheafy uniform affinoid ring. Instead, we give a sufficient condition for the sheaf property stronger than the uniformity in \S \ref{Affirmative Facts for Adic Spaces}. We remark that we always assume that the underlying f-adic ring of an affinoid ring is a complete Tate ring.

\subsection{Negative Facts}
\label{Negative Facts for Adic Spaces}

We give an example of strongly uniform affinoid ring such that its rational localisation is not uniform. The existence of such an example implies that the uniformisation of the affinoid algebra associated to a Banach $k$-algebra given in \cite{Ber} \S 1.3.\ (the strong uniformisation of a uniform affinoid ring given in Proposition \ref{uniformisation 2}) does not induce a uniformisation (strong uniformisation) of the adic spectrum.

\begin{thm}
\label{not uniform 2}
Let $r \in (1,\infty)$. Following the notation in \S \ref{Negative Facts for Berkovich Spaces}, set $Y \coloneqq \t{Spa}(k \set{r^{-a_i}U^iX^{a_i}}{i \in \N}^{\t{ad}})$ and $V \coloneqq \set{x \in Y}{\v{X(x)} \leq 1}$. The rational localisation $(\Opn_{Y}(V),\Opn_{Y}^{+}(V))$ is not a uniform affinoid ring over $k^{\t{ad}}$.
\end{thm}

\begin{proof}
By Proposition \ref{comparison}, $\Opn_{Y}(V)$ is isomorphic to the underlying topological $k$-algebra of $k \set{r^{-a_i}U^iX^{a_i}}{i \in \N} \ens{X}$. Therefore $(\Opn_{Y}(V),\Opn_{Y}^{+}(V))$ is not uniform by Theorem \ref{not uniform 1} and Proposition \ref{almost uniform - affinoid}, because any complete norm on the underlying topological $k$-algebra of a Banach $k$-algebra which gives its original topology is equivalent to its original norm by Banach's open mapping theorem (\cite{Bou}, Theorem I.3.3/1).
\end{proof}

We remark that the direct analogue of Corollary \ref{not sheafy 1} does not hold for the adic spectrum because the structure presheaf on an adic spectrum is well-defined by the universality of rational domains. Indeed, the rational domain $V' \coloneqq \set{x \in Y}{\v{X(x)} \leq 1, \v{(UX)(x)} \leq 1}$ does not coincides with $V$ unlike the corresponding rational domains of the Berkovich spectrum. Indeed, since $UX \in (\Opn_{Y}(V)^{\circ})^{\t{ac}} \backslash \Opn_{Y}(V)^{\circ}$, there is an $x \in V$ such that $\v{(UX)(x)} > 1$ by Lemma \ref{bounded ac - spectral radius}. They share points of height $1$, but $V$ possesses more points of general height than $V'$. This example does not seem to be so pathological, and hence we do not know whether it is sheafy or not. Using it, we construct a more pathological example.

\vspace{0.2in}
Let $(\A_i)_{i \in I}$ be a family of Banach $k$-algebras. There is an isometric homomorphism
\begin{eqnarray*}
  \left( \prod_{i \in I} \A_i \right) \Ens{T} & \hookrightarrow & \prod_{i \in I} \left( \A_i \Ens{T} \right) \\
  \sum_{h = 0}^{\infty} (a_{h,i})_{i \in I} T^h & \mapsto & \left( \sum_{h = 0}^{\infty} a_{h,i} T^h \right)_{i \in I},
\end{eqnarray*}
which is not surjective unless $\A_i \cong 0$ for all but finitely many $i \in I$. We identify the domain with the image. For an $(a_i)_{i \in I} \in \prod_{i \in I} \A_i$, it induces a bounded homomorphism $(\prod_{i \in I} \A_i) \ens{(a_i)_{i \in I}} \to \prod_{i \in I}(\A_i \Ens{a_i})$, which is submetric by the definition of the quotient norms. We remark that if $\A_i$ is uniform for any $i \in I$, the direct product $\prod_{i \in I}(\A_i \ens{T}) \to \prod_{i \in I}(\A_i \ens{a_i})$ of the canonical projections is surjective. Indeed, let $(f_i)_{i \in I} \in \prod_{i \in I} \A_i \ens{a_i}$ and $\epsilon > 0$. By the definition of the quotient norms, for any $i \in I$, there exists a lift $F_i \in \A_i \ens{T}$ of $f_i$ such that $\n{F_i}_{\A_i \ens{T}} \leq \n{f_i}_{\A_i \ens{a_i}} + \epsilon$. We have $\sup_{i \in I} \n{F_i}_{\A_i \ens{T}} \leq \sup_{i \in I} \n{f_i}_{\A_i \ens{a_i}} + \epsilon = \n{(f_i)_{i \in I}}_{\prod_{i \in I} \A_i \ens{a_i}} + \epsilon$. Therefore the sequence $(F_i)_{i \in I}$ lies in $\prod_{i \in I} \A_i \ens{T}$, and is a lift of $(f_i)_{i \in I}$.

\begin{lmm}
\label{uniform convergence}
Let $(\A_i)_{i \in I}$ be a family of Banach $k$-algebras, and $(a_i)_{i \in I} \in \prod_{i \in I} \A_i$ a sequence satisfying that there is a $C > 1$ such that for any $i \in I$ and $f_i \in \A_i$, the inequality $\n{a_i f_i} \geq C \n{f_i}$ holds. Then the canonical bounded homomorphism $(\prod_{i \in I} \A_i) \ens{(a_i)_{i \in I}} \to \prod_{i \in I}(\A_i \ens{a_i})$ is an isometry, and in particular, is injective.
\end{lmm}

\begin{proof}
Let $\varphi$ denote the canonical embedding $(\prod_{i \in I} \A_i) \ens{T} \hookrightarrow \prod_{i \in I}(\A_i \ens{T})$, and $\overline{\varphi}$ the induced homomorphism $(\prod_{i \in I} \A_i) \ens{(a_i)_{i \in I}} \to \prod_{i \in I}(\A_i \ens{a_i})$. Let $f \in (\prod_{i \in I} \A_i) \ens{T}$, and take a lift $F = \sum_{h = 0} F_h T^h \in (\prod_{i \in I} \A_i) \ens{T}$ of $f$. We prove $\n{f}_{(\prod_{i \in I} \A_i) \ens{T}} = \n{\overline{\varphi}(f)}_{\prod_{i \in I} (\A_i \ens{a_i})}$. Since $\overline{\varphi}$ is submetric, we have $\n{F}_{(\prod_{i \in I} \A_i) \ens{T}} \geq \n{f}_{(\prod_{i \in I} \A_i) \ens{(a_i)_{i \in I}}} \geq \n{\overline{\varphi}(f)}_{\prod_{i \in I} (\A_i \ens{a_i})}$. Therefore in order to verify $\n{f}_{(\prod_{i \in I} \A_i) \ens{(a_i)_{i \in I}}} = \n{\overline{\varphi}(f)}_{\prod_{i \in I} (\A_i \ens{a_i})}$, we may assume $\n{F}_{(\prod_{i \in I} \A_i) \ens{T}} > \n{\overline{\varphi}(f)}_{\prod_{i \in I} (\A_i \ens{a_i})}$. Put $F_h = (F_{h,i})_{i \in I} \in \prod_{i \in I} \A_i$ for each $h \in \N$. Since the composite of $\varphi$ and the canonical projection $\prod_{i \in I} (\A_i \ens{T}) \twoheadrightarrow \prod_{i \in I} (\A_i \ens{a_i})$ coincides with the composite of the canonical projection $(\prod_{i \in I} \A_i) \ens{T} \twoheadrightarrow (\prod_{i \in I} \A_i) \ens{(a_i)_{i \in I}}$ and $\overline{\varphi}$, the image $f_i \in \A_i \ens{a_i}$ of $\sum_{h = 0}^{\infty} F_{h,i} T^h \in \A_i \ens{T}$ satisfies $\n{f_i}_{\A_i \ens{a_i}} \leq \n{\overline{\varphi}(f)}_{\prod_{i \in I} (\A_i \ens{a_i})}$ for any $i \in I$.

\vspace{0.2in}
Let $\epsilon > 0$. For each $i \in I$, there is a $G_i \in \A_i \ens{T}$ with $\n{(T-a_i)G_i - \sum_{h = 0}^{\infty} F_{h,i} T^h}_{\A_i \ens{T}} < \n{\overline{\varphi}(f)}_{\prod_{i \in I} (\A_i \ens{a_i})} + \epsilon$ by $\n{f_i}_{\A_i \ens{a_i}} \leq \n{\overline{\varphi}(f)}_{\prod_{i \in I} (\A_i \ens{a_i})}$. We obtain $\n{G_i}_{\A_i \ens{T}} \leq \n{(T-a_i)G_i}_{\A_i \ens{T}} \leq \max \ens{\n{\overline{\varphi}(f)}_{\prod_{i \in I}(\A_i \ens{a_i})} + \epsilon, \n{F}_{(\prod_{i \in I} \A_i) \ens{T}}}$ by $\n{\sum_{h = 0}^{\infty} F_{i,h} T^h}_{\A_i \ens{T}} \leq \n{F}_{(\prod_{i \in I} \A_i) \ens{T}}$, and hence $G \coloneqq (G_i)_{i \in I}$ lies in $\prod_{i \in I} (\A_i \ens{T})$. We verify that $G$ lies in the image of $\varphi$. Put $G_i = \prod_{h = 0}^{\infty} G_{i,h} T^h$ for each $i \in I$. We have $R \coloneqq \overline{\lim}_{h \to \infty} \sup_{i \in I} \n{G_{i,h}}_{\A_i} \leq \n{G}_{\prod_{i \in I} (\A_i \ens{T})} < \infty$. Assume $R > 0$. Since $F \in (\prod_{i \in I} \A_i) \ens{T}$, there is an $h_0 \in \N$ such that $\n{F_h}_{\prod_{i \in I} \A_i} < R$ for any $h \geq h_0$. By the definition of $R$, for any $h_1 \in \N \backslash \ens{0}$, there is an $(h,i_0) \in \N \times I$ such that $h \geq h_0 + h_1$ and $\n{G_{i_0,h}} \geq R$. Then we have
\begin{eqnarray*}
  & & \nn{G_{i_0,h_0+h_1-1} - a_{i_0} G_{i_0,h_0+h_1}}_{\A_{i_0}} \leq \nn{\left( G_{i,h_0+h_1-1} - a_i G_{i+h_0+h_1} \right)_{\prod_{i \in I} \A_i}} = \nn{F_{h_0+h_1}}_{\prod_{i \in I} \A_i} < R \\
  & & \nn{a_{i_0} G_{i_0,h_0+h_1}}_{\A_{i_0}} \geq C \nn{G_{i_0,h_0+h_1}}_{\A_{i_0}} \geq CR,
\end{eqnarray*}
and hence $\n{G_{i_0,h_0+h_1-1}}_{\A_{i_0}} \geq CR > R$. Therefore we obtain $\n{G_{i_0,h_0+h_1-h}}_{\A_{i_0}} \geq C^{h} R > R$ by induction on $h$ for any $h \in \N \cap [0,h_1]$, and hence $\n{G}_{(\prod_{i \in I} \A_i) \ens{T}} \geq \n{G_{i_0,h_0}}_{\A_{i_0}} \geq C^{h_1} R$. Since $\lim_{h_1 \to \infty} C^{h_1} R = \infty$, it is a contradiction. We get $R = 0$. It implies that $\lim_{h \to \infty} \n{(G_{i,h})_{i \in I}}_{\prod_{i \in I} \A_i} = \overline{\lim}_{h \to \infty} \sup_{i \in I} \n{G_{i,h}}_{\A_i} = 0$, and hence $G$ lies in the image of $\varphi$. We conclude that
\begin{eqnarray*}
  & & \nn{f}_{\left( \prod_{i \in I} \A_i \right) \Ens{(a_i)_{i \in I}}} \leq \nn{F - (T - (a_i)_{i \in I}) G}_{\left( \prod_{i \in I} \A_i \right) \Ens{T}} \\
  & = & \sup_{i \in I} \nn{\left( \prod_{h = 0}^{\infty} F_{i,h} T^h \right) - (T-a_i) G_i}_{\A_i \Ens{T}} < \nn{\overline{\varphi}(f)}_{\prod_{i \in I} (\A_i \Ens{a_i})} + \epsilon \stackrel{\epsilon \to +0}{\longrightarrow} \nn{\overline{\varphi}(f)}_{\prod_{i \in I} (\A_i \Ens{a_i})},
\end{eqnarray*}
and thus $\overline{\varphi}$ is an isometry.
\end{proof}

We denote by $\t{C}_{\t{bd}}(\N,\R)$ the $R$-algebra of bounded $\R$-valued sequence.

\begin{lmm}
\label{direct product - localisation}
Let $r \in \t{C}_{\t{bd}}(\N,\R)$ with $r(\N) \subset (1,\infty)$. Then the canonical homomorphism
\begin{eqnarray*}
  \left( \prod_{n \in \N} k \Set{r(n)^{-a_i}U^iX^{a_i}}{i \in \N} \right) \Ens{(X)_{n \in \N}} \to \prod_{n \in \N} \left( k \Set{r(n)^{-a_i}U^iX^{a_i}}{i \in \N} \Ens{X} \right)
\end{eqnarray*}
is an isometry.
\end{lmm}

\begin{proof}
The given homomorphism is an isometry by Lemma \ref{uniform convergence}, because $\n{\cdot}_{k \set{r(n)^{-a_i}U^iX^{a_i}}{i \in \N}}$ is multiplicative and $\n{X}_{k \set{r(n)^{-a_i}U^iX^{a_i}}{i \in \N}} = r > 1$.
\end{proof}

\begin{thm}
\label{not uniform 3}
Let $r \in \t{C}_{\t{bd}}(\N,\R)$ with $r(\N) \subset (1,\infty)$. Set
\begin{eqnarray*}
  Y & \coloneqq & \t{Spa} \left( \left( \prod_{n \in \N} k \Set{r(n)^{-a_i}U^iX^{a_i}}{i \in \N} \right)^{\t{ad}} \right) \\
  V & \coloneqq & \set{x \in Y}{\v{(X)_{n \in \N}(x)} \leq 1}.
\end{eqnarray*}
The rational localisation $(\Opn_{Y}(V),\Opn_{Y}^{+}(V))$ is not a uniform affinoid ring over $k^{\t{ad}}$.
\end{thm}

\begin{proof}
It immediately follows from Proposition \ref{comparison}, Corollary \ref{circ - ac}, and Lemma \ref{direct product - localisation} because the proof of Theorem \ref{not uniform 2} shows $(UX)_{n \in \N} \in (\Opn_{Y}(V)^{\circ})^{\t{ac}} \backslash \Opn_{Y}(V)^{\circ}$.
\end{proof}

This new example is much more significant than $k \set{r^{-a_i}U^iX^{a_i}}{i \in \N} \ens{X}$ for study of Tate acyclicity. For any $f \in k \set{r^{-a_i}U^iX^{a_i}}{i \in \N} \ens{X}$, the multiplication of $T - f \in k \set{r^{-a_i}U^iX^{a_i}}{i \in \N} \ens{X} \ens{T}$ seems to be admissible. Indeed, $k \set{r(n)^{-a_i}U^iX^{a_i}}{i \in \N} \ens{X}$ is isometrically embedded in the completion $K \set{r^{-b_i}U^i}{i \in \N} \subset K[[U]]$ of $K[U]$ with respect to the norm
\begin{eqnarray*}
  \n{\cdot}_{K \Set{r^{-b_i}U^i}{i \in \N}} \colon K[U] & \to & [0,\infty) \\
  \sum_{i = 0}^{\infty} F_i U^i & \mapsto & \sup_{i \in \N} \v{F_i} r^{b_i}
\end{eqnarray*}
by the proof of Theorem \ref{not uniform 2}, where $K/k$ is the extension of complete valuation fields obtained as the completion of the fractional field of $k \ens{X}$ with respect to the Gauss norm of radius $1$. Therefore the computation of the admissibility is reduced to that for $K \set{r^{-b_i}U^i}{i \in \N} \ens{T}$, which is not so pathologic. Such an admissibility works well in the calculation of the $\check{\m{C}}$ech complex for Tate acyclicity, and hence $k \set{r^{-a_i}U^iX^{a_i}}{i \in \N}$ could be still sheafy. On the other hand, the same does not hold for the new example as is shown in the following.

\begin{prp}
\label{not admissible}
Suppose that $\v{k}$ is dense in $[0,\infty)$. Let $\pi \in \t{C}_{\t{bd}}(\N,k)$ be a sequence such that $(\v{\pi(n)})_{n \in \N} \in \t{C}_{\t{bd}}(\N,\R)$ is an increasing sequence converging to $1$. Then for any $r \in (1,\infty)$, the multiplication of $T - (\pi(n)UX)_{n \in \N}$ is not an admissible epimorphism on $(\prod_{n \in \N} k \set{r^{-a_i}U^iX^{a_i}}{i \in \N}) \ens{(X)_{n \in \N}} \ens{T}$.
\end{prp}

\begin{proof}
Put $\A \coloneqq k \set{r^{-a_i}U^iX^{a_i}}{i \in \N}$. By Lemma \ref{direct product - localisation} and the proof of Theorem \ref{not uniform 1}, the underlying $k$-algebra of $(\prod_{n \in \N} \A) \ens{(X)_{n \in \N}}$ is embedded in the underlying $k$-algebra of $\prod_{n \in \N} k[[X,U]]$, and hence $(\pi(n)UX)_{n \in \N}$ is not a zero divisor of $(\prod_{n \in \N} \A) \ens{(X)_{n \in \N}}$. Let $R > 0$. Since $(\v{\pi(n)})_{n \in \N}$ converges to $1$, there is an $(i_0,n_0) \in \N \times \N$ such that $\v{\pi(n_0)}^{i_0} r^{b_{i_0}} > R^{-1}$. Since $\rho_{\A}(\pi(n_0)UX) = \v{\pi(n_0)} < 1$ by the proof of Theorem \ref{not uniform 1}, $\pi(n_0)UX$ is topologically nilpotent and hence there is an $N > i_0$ such that $\n{(\pi(n_0)UX)^N} \leq 1$. Denote by $e_{n_0} \in (\prod_{n \in \N} \A) \ens{(X)_{n \in \N}}$ the image of $1 \in \A \ens{X}$ by the zero-extension $\A \ens{X} \hookrightarrow \prod_{n \in \N} (\A \ens{X})$ outside the $n_0$-th entry. It lies in the image of $(\prod_{n \in \N} \A) \ens{X}$. Then we have $\n{\sum_{i = 0}^{N-1} (\pi(n_0)UX)^{N-1-i} e_{n_0} T^i}_{(\prod_{n \in \N} \A) \ens{X} \ens{T}} \geq \n{(\pi(n_0)UX)^{i_0}}_{\A \ens{X}} > R^{-1}$, and hence
\begin{eqnarray*}
  & & \nn{\left( T - (\pi(n)UX)_{n \in \N} \right) \sum_{i = 0}^{N-1} (\pi(n_0)UX)^{N-1-i} e_{n_0} T^i}_{\left( \prod_{n \in \N} \A \right) \ens{X} \ens{T}} \\
  & = & \nn{e_{n_0}T^N - (\pi(n_0)UX)^N e_{n_0}}_{\left( \prod_{n \in \N} \A \right) \ens{X} \ens{T}} = 1.
\end{eqnarray*}
It implies that the inverse $(T - (\pi(n)UX)_{n \in \N})^{-1}$ of the bounded bijective $k$-linear homomorphism $(\prod_{n \in \N} k \set{r^{-a_i}U^iX^{a_i}}{i \in \N}) \ens{(X)_{n \in \N}} \ens{T} \to (T - (\pi(n)UX)_{n \in \N})$ given by the multiplication of $T - (\pi(n)UX)_{n \in \N}$ has operator norm greater than $R$. Thus $(T - (\pi(n)UX)_{n \in \N})^{-1}$ is not bounded, and hence the multiplication of $T - (\pi(n)UX)_{n \in \N}$ is not admissible.
\end{proof}

The proof of Proposition \ref{not admissible} implies that $(\pi(n)UX)_{n \in \N}$ is not bounded, while $\pi(n)UX$ is topologically nilpotent for any $n \in \N$. In other word, topologically nilpotent elements converges to an almost bounded element in the strong topology. This phenomenon help us to verify that the $\check{\m{C}}$ech complex for Tate acyclicity is not exact. Finally we achieve the main result.

\begin{thm}
\label{not sheafy}
Suppose that $\v{k}$ is dense in $[0,\infty)$. For any $r \in (1,\infty)$, the strongly uniform affinoid ring $(\prod_{n \in \N} k \set{r^{-a_i}U^iX^{a_i}}{i \in \N})^{\t{ad}}$ over $k^{\t{ad}}$ is not sheafy.
\end{thm}

\begin{proof}
Take a $\pi \in \t{C}_{\t{bd}}(\N,k)$ such that $\v{\pi(\N)} \subset (\sqrt[]{\mathstrut r},1)$ and $(\v{\pi(n)})_{n \in \N} \in \t{C}_{\t{bd}}(\N,\R)$ is an increasing sequence converging to $1$. Put
\begin{eqnarray*}
  & &
  \begin{array}{rclcrcl}
    \A & \coloneqq & k \set{r^{-a_i}U^iX^{a_i}}{i \in \N}, & & \X & \coloneqq & (X)_{n \in \N} \in \B \\
    \B & \coloneqq & \prod_{n \in \N} \A, & & \U \X & \coloneqq & (UX)_{n \in \N} \in \B
  \end{array} \\
  & & \Pi \coloneqq \pi = (\pi(n))_{n \in \N} \in \t{C}_{\t{bd}}(\N,k) \hookrightarrow \t{C}_{\t{bd}}(\N,\A) = \B.
\end{eqnarray*}
It suffices to verify that the $\check{\m{C}}$ech complex
\begin{eqnarray*}
  0 \to \B \Ens{\X} \to \prod_{\sigma \in \Ens{\pm 1}} \B \Ens{\X} \Ens{(\Pi \U \X)^{\sigma}} \to \B \ens{\X} \ens{(\Pi \U \X)^{\pm 1}}
\end{eqnarray*}
is not exact. Since $(\v{\pi(n)})_{n \in \N}$ is an increasing sequence, so is $(\sup_{i \in \N} \v{\pi(n)}^i r^{b_i})_{n \in \N}$. Put $R_n \coloneqq \sup_{i \in \N} \v{\pi(n)}^i r^{b_i}$ for each $n \in \N$. Note that $(R_n)_{n \in \N} \in (1,\infty)^{\N}$ because $(\v{\pi(n)})_{n \in \N} \in (0,1)^{\N}$ and $(b_i)_{i \in \N} \in O(\log i)$. For each $n \in \N$, denote by $i_n \in \N \backslash \ens{0}$ the smallest integer such that $\v{\pi(n)}^{i_n} r^{b_{i_n}} = R_n$, by $l_n \in \N$ the greatest integer such that $\v{\pi(n)}^{-l_n} \leq \sqrt[]{\mathstrut R_n}$, and by $N_n \in \N$ the smallest integer such that $\v{\pi(n)}^{N_n} R_n \leq 1$ and $\v{\pi(n)}^{N_n} r \leq 1$. Since $\v{\pi(n)}^2 \geq \v{\pi(0)}^2 > r$, $l_n \geq 1$ for any $n \in \N$. Set
\begin{eqnarray*}
  F \coloneqq \left( \pi(n)^{l_n} \sum_{i = -N_n}^{i_n} (\pi(n)UX)^{i_n - i} T^i \right)_{n \in \N} \in \prod_{n \in \N} \left( k[[U,X,T]][T^{-1}] \right).
\end{eqnarray*}
Since $\v{\pi(n)}^{-1} \sqrt[]{\mathstrut R_n} \leq \v{\pi(n)}^{i_n}(\pi(n)UX)^{i_n} \leq \sqrt[]{\mathstrut R_n}$ for any $n \in \N$ by definition, $F$ does not lie in the image of $\B \ens{\X} \ens{T,T^{-1}}$. We have
\begin{eqnarray*}
  (T - \Pi \U \X) F = \left( \pi(n)^{l_n} \left( T^{i_n + 1} - (\pi(n)UX)^{N_n + i_n + 1} T^{-N_n} \right) \right)_{n \in \N}.
\end{eqnarray*}
For each $n \in \N$, denote by $e_n \in \B$ the image of $1 \in \A$ by the zero-extension $\A \hookrightarrow \B$ outside the $n$-th entry. Since $\v{\pi(n)}^{l_n} \leq \v{\pi(n)}^{-1} \sqrt[]{\mathstrut R_n}^{\ -1} \stackrel{n \to \infty}{\longrightarrow} 0$, $(T - \Pi \U \X) F$ lies in the image of $\B \ens{\X} \ens{T,T^{-1}}$, and as an element of $\B \ens{\X} \ens{T,T^{-1}}$, it is the limit of the sequence
\begin{eqnarray*}
  \left( (T - \Pi \U \X) F \sum_{l = 1}^{m} e_l \right)_{m \in \N} = \left( \sum_{n = 1}^{m} \pi(n)^{l_n} \left( e_n T^{i_n + 1} - (\pi(n)UX)^{N_n + i_n + 1} e_n T^{-N_n} \right) \right)_{m \in \N}
\end{eqnarray*}
in $\B \ens{\X} \ens{T,T^{-1}}$. Each entry is an element of $(T - \Pi \U \X) \subset \B \ens{\X} \ens{T,T^{-1}}$ by the computation in the proof of Theorem \ref{not admissible}, and hence we have $(T - \Pi \U \X) F \in (T - \Pi \U \X)^{\hat{}} \subset \B \ens{\X} \ens{T,T^{-1}}$. Set $G_{+} \coloneqq \sum_{n = 0}^{\infty} \pi(n)^{l_n} e_n T^{i_n + 1} \in \B \ens{\X} \ens{T}$ and $G_{-} \coloneqq \sum_{n = 0}^{\infty} \pi(n)^{l_n} (\pi(n)UX)^{N_n + i_n + 1} e_n T^{-N_n} \in \B \ens{\X} \ens{T^{-1}}$. Then $G_{+} - G_{-} = (T - \Pi \U \X)F$, and hence $G_{+}(\Pi \U \X) = G_{-}((\Pi \U \X)^{-1})$ as elements of $\B \ens{\X} \ens{(\Pi \U \X)^{\pm 1}}$. We verify that $G_{+}(\Pi \U \X)$ does not lie in the image of $\B \ens{\X}$ in $\B \ens{\X} \ens{\Pi \U \X}$. Let $f = (f_n)_{n \in \N} \in \B \ens{\X}$, and assume that $G_{+} - f$ lies in $(T - \Pi \U \X)^{\hat{}} \subset \B \ens{\X} \ens{T}$. Set $F_{+} \coloneqq (\pi(n)^{l_n} \sum_{i = 0}^{i_n} (\pi(n)UX)^{i_n - i} T^i)_{n \in \N} \in \prod_{n \in \N} k[[U,X,T]]$. We have $(T - \Pi \U \X)F_{+} = (\pi(n)^{l_n}(T^{i_n + 1} - (\pi(n)UX)^{i_n + 1}))_{n \in \N} = G_{+} - (\pi(n)^{l_n} (\pi(n)UX)^{i_n + 1})_{n \in \N} \in \B \ens{\X} \ens{T}$. By $G_{+} - f \in (T - \Pi \U \X)^{\hat{}}$, there is an $H = \sum_{h = 0}^{\infty} (H_{h,n})_{n \in \N} T^h \in \B \ens{\X} \ens{T}$ with $\n{G_{+} - f - (T - \Pi \U \X)H}_{\B \ens{T}} < \v{\pi(n)}^{N_n + l_n} R_n$. By $\n{\pi(n)^{l_n}(\pi(n)UX)^{i_n + 1})}_{\A} \geq \v{\pi(n)}^{l_n + 1} R_n \geq \v{\pi(n)} \sqrt[]{\mathstrut R_n} \stackrel{n \to \infty}{\longrightarrow} \infty$, there is an $n \in \N$ such that $\n{\pi(n)^{l_n}(\pi(n)UX)^{i_n + 1})}_{\A \ens{X}} > \max \ens{\n{f}_{\B \ens{\X}}, \n{H}_{\B \ens{\X} \ens{T}}}$ and $l_n > 1$. Set $H_n \coloneqq \sum_{h = 0}^{\infty} H_{h,n} T^h \in \A \ens{X} \ens{T}$, $F_{+,n} \coloneqq \pi(n)^{l_n} \sum_{i = 0}^{i_n} (\pi(n)UX)^{i_n - i} T^i \in \A \ens{X}[T]$, and $H_{h,n} = \sum_{i = 0}^{\infty} \sum_{j = b_i}^{\infty} H_{h,n,i,j} U^i X^j$ for each $h \in \N$. We obtain
\begin{eqnarray*}
  & & \nn{\pi(n)^{l_n} (\pi(n)UX)^{i_n + 1} - f_n + (T - \pi(n)UX)(F_{+,n} - H_n)}_{\A \Ens{X} \Ens{T}} \\
  & = & \nn{\pi(n)^{l_n} T^{i_n + 1} - f_n - (T - \pi(n)UX) H_n}_{\A \Ens{X} \Ens{T}} \\
  & \leq & \nn{G_{+} - f - (T - \Pi \U \X)H}_{\B \Ens{\X} \Ens{T}} < \vv{\pi(n)}^{N_n + l_n} R_n.
\end{eqnarray*}
It implies the inequalities
\begin{eqnarray*}
  & & \nn{1 - H_{i_n,n} + \pi(n)UX H_{i_n+1,n}}_{\A \Ens{X}} < \vv{\pi(n)}^{N_n + l_n} R_n \\
  & & \nn{- H_{h,n} + \pi(n)UX H_{h+1,n}}_{\A \Ens{X}} < \vv{\pi(n)}^{N_n + l_n} R_n \\
  & & \nn{\pi(n)^{l_n} (\pi(n)UX)^{i_n + 1} - f_n + (\pi(n)UX)^{i_n + 1} - (\pi(n)UX) H_{0,n})}_{\A \Ens{X}} < \vv{\pi(n)}^{N_n + l_n} R_n
\end{eqnarray*}
hold for any $h \in \N \backslash \ens{i_n}$. In particular, we get $\v{1 - H_{N_n,n,0,0}} \leq \v{\pi(n)}^{N_n + l_n} R_n \leq \v{\pi(n)}^{l_n} < 1$ and $\v{H_{N_n,n,0,0}} = 1$. Considering the coefficient of $U^i X^i$ in $- H_{i_n - i,n} + \pi(n)UX H_{i_n - i + 1,n}$ inductively on $i \in \N \cap [1,i_n]$, we obtain $\n{(\pi(n)UX)^{i_n} - (UX)^{i_n} H_{0,n,i_n,i_n}}_{\A \ens{X}} < \v{\pi(n)}^{N_n + l_n} R_n < R_n = \n{(\pi(n)UX)^{i_n}}_{\A \ens{X}}$ and $\n{(UX)^{i_n} H_{0,n,i_n,i_n}}_{\A \ens{X}} = \n{(\pi(n)UX)^{i_n}}_{\A \ens{X}} = R_n$. Therefore we get
\begin{eqnarray*}
  & & \nn{(\pi(n)UX)^{i_n + 1} - (\pi(n)UX) H_{0,n,i_n,i_n})}_{\A \Ens{X}} \\
  & \leq & \nn{\pi(h)UX}_{\A \Ens{X}} \nn{(\pi(n)UX)^{i_n} - (UX)^{i_n} H_{0,n,i_n,i_n}}_{\A \Ens{X}} < \vv{\pi(n)}^{N_n + l_n + 1} r R_n \\
  & \leq & \vv{\pi(n)}^{l_n + 1} R_n = \vv{\pi(n)}^{l_n + 1} \nn{(\pi(n)UX)^{i_n}}_{\A \Ens{X}} \leq \nn{\pi(n)^{l_n} (\pi(n)UX)^{i_n + 1}}_{\A \Ens{X}}
\end{eqnarray*}
and hence
\begin{eqnarray*}
  & & \nn{\pi(n)^{l_n} (\pi(n)UX)^{i_n + 1} - f_n + (\pi(n)UX)^{i_n + 1} - (\pi(n)UX) H_{0,n})}_{\A \Ens{X}} \\
  & \geq & \nn{\pi(n)^{l_n} (\pi(n)UX)^{i_n + 1}}_{\A \Ens{X}} \geq \vv{\pi(n)}^{l_n + 1} R_n
\end{eqnarray*}
because $\n{f_n}_{\A \ens{X}} \leq \n{f}_{\B} < \nn{\pi(n)^{l_n} (\pi(n)UX)^{i_n + 1}}_{\A \Ens{X}}$. It contradicts $\n{\pi(n)^{l_n} (\pi(n)UX)^{i_n + 1} - f_n + (\pi(n)UX)^{i_n + 1} - (\pi(n)UX) H_{0,n})}_{\A \ens{X}} < \v{\pi(n)}^{N_n + l_n} R_n$. Thus $G_{+}(\Pi \U \X)$ does not lie in the image of $\B \Ens{\X}$. We conclude that the sequence
\begin{eqnarray*}
  0 \to \B \Ens{\X} \to \prod_{\sigma \in \Ens{\pm 1}} \B \Ens{\X} \Ens{(\Pi \U \X)^{\sigma}} \to \B \ens{\X} \ens{(\Pi \U \X)^{\pm 1}}
\end{eqnarray*}
is not exact, and $\B^{\t{ad}}$ is not sheafy.
\end{proof}

\subsection{Affirmative Facts}
\label{Affirmative Facts for Adic Spaces}

In this subsection, let $K$ denote a topological field whose topology is given by a complete valuation of height $1$. Such a complete valuation of height $1$ is unique up to equivalences because its valuation ring coincides with $O_k \subset K$, and we fix a valuation $\v{\cdot} \colon K \to [0,\infty)$ of height $1$ corresponding to $O_k$ so that we can apply results of a complete valuation field. We deal with several affirmative facts deeply related to perfectoid theory. For details about perfectoid theory, see \cite{Sch}.

\begin{dfn}
\label{locally uniform}
An affinoid ring $\A$ over $K$ is said to be {\it locally uniform} if for any rational domain $U \subset X \coloneqq \t{Spa}(\A)$, the affinoid ring $(\Opn_{X}(U),\Opn_{X}^{+}(U))$ over $K$ is uniform.
\end{dfn}

The local uniformity implies the uniformity because the completeness of $\A^{\triangleright}$ implies $\A \cong (\Opn_{X}(X),\Opn_{X}^{+}(X))$ by the universality of a rational domain (\cite{Hub2} (1.2)). It is a local property by definition in the sense that an affinoid ring $\A$ is locally uniform if and only if every rational covering $\U$ of $X = \t{Spa}(\A)$ satisfies that $(\Opn_{X}(U),\Opn_{X}^{+}(U))$ is locally uniform for any $U \in \U$. We remark that it is strictly stronger than the uniformity by Theorem \ref{not uniform 2}. The class of locally uniform affinoid rings contains all the affinoid algebras associated to reduced Banach $K$-algebras topologically of finite type and perfectoid affinoid algebras over perfectoid fields.

\begin{thm}[Tate's Acyclicity]
\label{acyclicity 3}
Let $\A$ be a locally uniform affinoid ring over $K$. Then for any rational covering $\U$ of $X = \t{Spa}(\A)$, the $\check{\m{C}}$ech complex
\begin{eqnarray*}
  0 \to \A^{\triangleright} \to \prod_{U \in \U}\t{H}^0(U, \Opn_X) \to \prod_{U,U' \in \U}\t{H}^0(U \cap U', \Opn_X) \to \cdots
\end{eqnarray*}
is an exact sequence of complete topological $K$-vector spaces such that the images endowed with the relative topologies are canonically homeomorphic to the corresponding cokernels endowed with the quotient topology.
\end{thm}

The following is just imitating the proof of \cite{BGR} Proposition 8.2.2/5.

\begin{proof}
For each uniform affinoid ring $\B$ over $K$, we equip $\B^{\triangleright}$ with $\n{\cdot}_{\B} \coloneqq \rho_{\B}$ introduced in \S \ref{Affinoid Rings}, and regard it as a uniform Banach $K$-algebra. We deal with rational localisations of adic spectra comparing those of Berkovich spectra by Proposition \ref{comparison}. The statement on the topologies of the images and the cokernels is equivalent to the admissibility of the sequence as a complex of Banach $K$-vector spaces. Since $\A$ is locally uniform, $\A^{\triangleright}$ is a strongly Banach function algebra over $k$ with respect to a norm which gives its original topology by Corollary \ref{spectral radius} and Proposition \ref{comparison}. We remark that for any rational domain $V \subset X$, the restriction of a cyclic covering, a strongly cyclic covering, or a special covering (Definition \ref{special}) of $X$ to $V$ is a ration covering of $V$ of the same type. If $\U$ is special, then the assertion follows from Theorem \ref{acyclicity 2}. Hence if $\U$ is strongly cyclic, then the assertion follows from Lemma \ref{refinement} (iii) and \cite{BGR} Corollary 8.1.4/3. Therefore if $\U$ is a cyclic covering, then the assertion follows from Lemma \ref{refinement} (ii) and \cite{BGR} Corollary 8.1.4/3. Consequently, for a general rational covering $\U$, the assertion follows from Lemma \ref{refinement} (i) and \cite{BGR} Corollary 8.1.4/3.
\end{proof}

\begin{thm}
\label{sheafy}
Every locally uniform affinoid ring over $K$ is sheafy.
\end{thm}

We remark that Kevin Buzzard and Alain Verberkmoes worked independently on this problem, and found the same result in \cite{BV}.

\begin{proof}
Since the local uniformity is a local property, it suffices to verify the exactness of the $\check{\m{C}}$ech complex of sections associated to a rational covering of the total space, and it directly follows from Theorem \ref{acyclicity 3}.
\end{proof}

Theorem \ref{sheafy} gives an alternative proof of the sheaf property of a perfectoid affinoid algebra over a perfectoid field. Peter Scholze verified the local uniformity in \cite{Sch} Corollary 6.8.\ before the proof of the sheaf condition in \cite{Sch} Proposition 6.14. The proof of the sheaf condition is done in several steps containing the tilting technique for reducing it to the case where the characteristic of the base field is positive. On the other hand, this alternative proof works directly for any characteristic. Now we introduce a notion of almost acyclicity. For details of almost mathematics, see \cite{GR}.

\begin{dfn}
Suppose that $\v{K}$ is dense in $[0,\infty)$. For an affinoid ring $\A$ over $K$, a finite rational covering $\set{\t{Spa}(R_i,R_i^{+}) \hookrightarrow X}{i \in I}$ of $X \coloneqq \t{Spa}(\A)$ is said to be {\it almost $2$-acyclic} if the $\check{\m{C}}$ech complex
\begin{eqnarray*}
  0 \to (\A^{\triangleright})^{\circ} \to \prod_{i_1 \in I} R_{i_1}^{\circ} \to \prod_{i_1,i_2 \in I} \left( R_{i_1} \hat{\otimes}_K R_{i_2} \right)^{\circ} \to \prod_{i_1,i_2,i_3 \in I} \left( R_{i_1} \hat{\otimes}_K R_{i_2} \hat{\otimes}_K R_{i_3} \right)^{\circ}
\end{eqnarray*}
is almost exact as a complex of $O_k$-modules.
\end{dfn}

We remark that even if the image of $O_k$ is not contained in $\A^{+}$, it is contained in $\A^{\circ}$ because every bounded homomorphism sends a bounded subset to a bounded subset. For an almost $2$-acyclic rational covering $\set{U_i \hookrightarrow \t{Spa}(\A)}{i \in I}$ of $\t{Spa}(\A)$, the $\check{\m{C}}$ech complex
\begin{eqnarray*}
  0 \to \A^{\triangleright} \to \prod_{i_1 \in I} R_{i_1} \to \prod_{i_1,i_2 \in I} R_{i_1} \hat{\otimes}_K R_{i_2} \to \prod_{i_1,i_2,i_3 \in I} R_{i_1} \hat{\otimes}_K R_{i_2} \hat{\otimes}_K R_{i_3}
\end{eqnarray*}
is exact because $K$ is flat over $O_k$ and every divisible almost zero $O_k$-module is zero. In the following, suppose that $K$ is a perfectoid field.

\begin{thm}
\label{locally perfectoid}
Let $\A$ be a sheafy affinoid ring over $K$. Then the following are equivalent:
\begin{itemize}
\item[(i)] The affinoid ring $\A$ is a perfectoid affinoid $K$-algebra.
\item[(ii)] There is a rational covering of $\t{Spa}(\A)$ consisting of affinoid perfectoid spaces, and every rational covering of $\t{Spa}(\A)$ is almost $2$-acyclic.
\item[(iii)] There is an almost $2$-acyclic rational covering of $\t{Spa}(\A)$ consisting of affinoid perfectoid spaces.
\end{itemize}
\end{thm}

\begin{proof}
The condition (i) implies (ii) by \cite{Sch} Proposition 6.14. The condition (ii) immediately implies (iii). Therefore it suffices to verify that (iii) implies (i). Let $\A$ be a sheafy affinoid ring over $K$ such that $X = \t{Spa}(\A)$ admits an almost $2$-acyclic rational covering $\set{U_i \subset X}{i \in I}$ by affinoid perfectoid spaces $U_i = \t{Spa}(S_i,S_i^{+})$. Put
\begin{eqnarray*}
  C & = & ((C(-1),C(0),C(1),C(2),C(3)),d^{\bullet}) \\
  & \coloneqq & \left( 0 \to \A^{\triangleright} \to \prod_{i_1 \in I} S_{i_1} \to \prod_{i_1,i_2 \in I} S_{i_1} \hat{\otimes}_k S_{i_2} \to \prod_{i_1,i_2,i_3 \in I} S_{i_1} \hat{\otimes}_k S_{i_2} \hat{\otimes}_k S_{i_3} \right) \\
  \tau C & = & ((C(-1),C(0),C(1),C(2)),d^{\bullet}) \coloneqq \left( 0 \to \A^{\triangleright} \to \prod_{i_1 \in I} S_{i_1} \to \prod_{i_1,i_2 \in I} S_{i_1} \hat{\otimes}_k S_{i_2} \right)
\end{eqnarray*}
Then $C^{\circ} \subset C$ is almost exact by the assumption, and hence $C$ is exact. Since $C^{\circ}$ is a subcomplex of $C$, it is exact at $C(0)^{\circ}$. By Banach's open mapping theorem (\cite{Bou} Theorem I.3.3/1), the exactness of $C$ at $C(1)$ implies that $C(0)$ is isomorphic to $\ker(d^1)$, and hence $C(0)^{\circ}$ is isomorphic to $\ker(d^1)^{\circ} = \ker(d^1) \cap C(1)^{\circ}$, where $\ker(d^1)$ is regarded as a Banach $k$-algebra because it is a Banach $k$-subalgebra of $C(1)$. Therefore $C^{\circ}$ is exact at $C(1)^{\circ}$, and $(\A^{\triangleright})^{\circ} = C(0)^{\circ}$ is bounded again by Banach's open mapping theorem. Since the functor $(\cdot)^{\t{a}} \colon M \rightsquigarrow M^{\t{a}}$ is exact, $C^{\circ \t{a}}$ is exact. Let $p > 0$ denote the characteristic of the residue field of $K$ with respect to the valuation ring $O_k$. Take a uniformiser $\varpi \in K$ with $\v{p} < \v{\varpi} < 1$ which admits a $p$-power root $\varpi^{1/p} \in K$. Taking the cohomology of the exact sequence $0 \to C^{\circ \t{a}} \stackrel{\varpi}{\longrightarrow} C^{\circ \t{a}} \to C^{\circ \t{a}}/\varpi \to 0$, we obtain $\t{H}^i(\tau C^{\circ \t{a}}/\varpi) = 0$ for $i = 0,1$ by the almost exactness of $\tau C^{\circ}$. Thus $(\tau C^{\circ \t{a}})/\varpi$ is exact.

\vspace{0.2in}
We obtained the exactness of $\tau C^{\circ \t{a}}/\varpi$. Replacing $\varpi$ by $\varpi^{1/p}$, we also obtain the exactness of $\tau C^{\circ \t{a}}/\varpi^{1/p}$. Consider the diagram
\begin{eqnarray*}
\begin{CD}
  0 @>>> C(0)^{\circ \t{a}}/\varpi^{1/p} @>>> C(1)^{\circ \t{a}}/\varpi^{1/p} @>>> C(2)^{\circ \t{a}}/\varpi^{1/p} \\
  @.     @VVV                                 @VVV                                 @VVV \\
  0 @>>> C(0)^{\circ \t{a}}/\varpi       @>>> C(1)^{\circ \t{a}}/\varpi       @>>> C(2)^{\circ \t{a}}/\varpi,
\end{CD}
\end{eqnarray*}
where each vertical arrow is the Frobenius homomorphism. Since each component of $C(1)$ and $C(2)$ is a perfectoid $K$-algebra, the second and third vertical arrows are isomorphisms by \cite{Sch} Proposition 5.5. Therefore the first vertical arrow is an isomorphism by five lemma. It implies that $C(0)^{\circ \t{a}}$ is a perfectoid $K^{\circ \t{a}}$-algebra, and hence $R = C(0)$ is a perfectoid $K$-algebra by \cite{Sch} Lemma 5.6.
\end{proof}

\begin{crl}
Suppose that $K$ is of characteristic $p > 0$. Then a perfectoid space over $K$ is an affinoid perfectoid space if and only if it is an affinoid space.
\end{crl}

We note that many people simultaneously worked on this problem. For example, Yoichi Mieda gave me a proof in personal communication directly verifying the uniformity and the perfectness of the global section. Kevin Buzzard and Alain Verberkmoes also gave a proof under the assumption of the local uniformity in \cite{BV}.

\begin{proof}
The necessary implication is obvious. For the sufficient implication, it is reduced to the case where an affinoid space admits a covering of the form in Corollary \ref{acyclicity} consisting of two affinoid perfectoid spaces by the same argument in Theorem \ref{sheafy}. Let $X = \t{Spa}(\A)$ be a perfectoid space which is an affinoid space, and $f \in \A^{\triangleright}$ an element such that $\A^{\triangleright} \ens{f}$ and $\A^{\triangleright} \ens{f^{-1}}$ are perfectoid $K$-algebras. The $\check{\m{C}}$ech complex
\begin{eqnarray*}
  C & = & ((C(-1),C(0),C(1),C(2),C(3)),d^{\bullet}) \\
  & \coloneqq & \left( 0 \to \A^{\triangleright} \to \A^{\triangleright} \Ens{f} \times \A^{\triangleright} \Ens{f^{-1}} \to \A^{\triangleright} \Ens{f,f^{-1}} \to 0 \right)
\end{eqnarray*}
is exact by Corollary \ref{acyclicity}. Therefore by the same argument as in the proof of Theorem \ref{locally perfectoid}, $C^{\circ}$ is almost exact at $C(0)^{\circ}$ and $C(1)^{\circ}$. By Banach's open mapping theorem (\cite{Bou} Theorem I.3.3/1), $d^1(C(1)^{\circ})$ is an open $O_k$-submodule of $C(2)^{\circ}$, and hence there is an $r \in O_k \backslash \ens{0}$ such that $r C(2)^{\circ} \subset d^1(C(1)^{\circ})$. Let $f \in C(2)^{\circ}$. Take an $\epsilon \in O_k$ with $\v{\epsilon} < 1$. We have $\v{\epsilon^{p^N}} < \v{r}$ for a sufficiently large $N \in \N$, and hence $(\epsilon f)^{p^N} \in d^1(C(1)^{\circ})$. Since the underlying $K$-algebra of $C(1)$ is perfect and $d^1$ commutes with Frobenius, $\epsilon f \in d^1(C(1)^{\circ})$. Thus $d^1 \colon C(1)^{\circ} \to C(2)^{\circ}$ is almost surjective, and $C^{\circ}$ is almost exact at $C(2)^{\circ}$. We conclude that the rational covering is almost $2$-acyclic, and hence $X$ is an affinoid perfectoid space by Theorem \ref{locally perfectoid}.
\end{proof}

\vspace{0.4in}
\addcontentsline{toc}{section}{Acknowledgements}
\noindent {\Large \bf Acknowledgements}
\vspace{0.1in}

I am deeply grateful to Takeshi Tsuji for discussions and advices in seminars. I thank to Yoichi Mieda for instructing me on elementary facts on adic spaces. I appreciate to Naoki Imai for giving significant opinions on this paper. I express my gratitude to Peter Scholze and Kevin Buzzard for giving me several interesting suggestions in personal communication. I am thankful to my friends for giving me an opportunity to talk about this topic and sharing so much time with me to discuss. I express my gratitude to family for their deep affection.

\vspace{0.2in}
I am a research fellow of Japan Society for the Promotion of Science. This work was supported by the Program for Leading Graduate 
Schools, MEXT, Japan.


\addcontentsline{toc}{section}{References}
\begin{thebibliography}{99}

\bibitem[Ber]{Ber} V.\ G.\ Berkovich, {\it Spectral Theory and Analytic Geometry over non-Archimedean Fields}, Mathematical Surveys and Monographs, Number 33, the American Mathematical Society, 1990.

\bibitem[BGR]{BGR} S.\ Bosch, U.\ G\"untzer, and R.\ Remmert, {\it Non-Archimedean Analysis A Systematic Approach to Rigid Analytic Geometry}, Grundlehren der mathematischen Wissenschaften 261, A Series of Comprehensive Studies in Mathematics, Springer, 1984.

\bibitem[Bou]{Bou} N.\ Bourbaki, {\it Espaces Vectoriels Topologiques}, El\'ements de math\'ematique, vol.\ V, Hermann, 1953.

\bibitem[BV]{BV} K.\ Buzzard, A.\ Verberkmoes, {\it Stably Uniform Affinoids are Sheafy}, preprint.

\bibitem[GR]{GR} O.\ Gabber, L.\ Ramero, {\it Almost Ring Theory}, Lecture Notes in Mathematics, Springer, 2003.

\bibitem[Hub1]{Hub1} R.\ Huber, {\it Continuous Valuations}, Mathematische Zeitschrift, Volume 212, pp.\ 455--477, Springer, 1993.

\bibitem[Hub2]{Hub2} R.\ Huber, {\it A Generalization of Formal Schemes and Rigid Analytic Varieties}, Mathematische Zeitschrift, Volume 217, pp.\ 513--551, Springer, 1994.

\bibitem[Hub3]{Hub3} R.\ Huber, {\it \'Etale Cohomology of Rigid Analytic Varieties and Adic Spaces}, Aspects of Mathematics, Vol.\ 30, Friedrich Vieweg \& Sohn Verlagsgesellschaft mbH, Braunschweig/Wiesbaden, 1996.

\bibitem[Mih]{Mih} T.\ Mihara, {\it Characterisation of the Berkovich Spectrum of the Banach Algebra of Bounded Continuous Functions}, Documenta Mathematica, Volume 19, pp.\ 769--799, 2014.

\bibitem[Sch]{Sch} P.\ Scholze, {\it Perfectoid Spaces}, Publications Math\'ematiques de I'H\'ES, Volume 116, Issue 1, pp.\ 245--313, Springer, 2012.

\bibitem[Tat]{Tat} J.\ Tate, {\it Rigid Analytic Spaces}, Inventiones Mathematicae, Volume 12, Issue 4, pp.\ 257--289, Springer, 1971.

\end{thebibliography}
\end{document}